\documentclass[12pt]{amsart}
\usepackage{amsmath}
\usepackage{amssymb}
\usepackage{enumerate} 
\usepackage{graphicx, tikz} 
\usepackage{hyperref}

\newtheorem{thm}{Theorem}[section]

\newtheorem{hypo}[thm]{Hypothesis}

\newtheorem{prop}[thm]{Proposition}

\newtheorem{rmk}[thm]{Remark}

\title[Birkhoff sums and local times of the periodic Lorentz gas in infinite horizon]{Limit theorems for Birkhoff sums and local times of the periodic Lorentz gas with infinite horizon}
\author[Fran\c{c}oise P\`ene]{
	{Fran\c{c}oise P\`ene}
	\address{Univ Brest, Universit\'e de Brest, LMBA,
		 CNRS UMR 6205, 
		6 avenue Le Gorgeu, 29238 Brest cedex, France}
}

\begin{document}

\begin{abstract}
	This work is a contribution to the study of the ergodic and stochastic properties of $\mathbb Z^d$-periodic dynamical systems preserving an infinite measure. 
We establish functional limit theorems for natural Birkhoff sums
related to local times 
of the $\mathbb Z^d$-periodic Lorentz gas with infinite horizon, for both the collision map and the flow.
In particular, our results apply to 
the difference between the numbers of collisions in two different cells. 
Because of the $\mathbb Z^d$-periodicity of the model we are interested in, these Birkhoff sums can be rewritten as additive functionals of a Birkhoff sum of the Sinai billiard. 
Our proofs rely on a general argument valid in a general framework.
For completness and in view of future studies, we state a general result
of convergence of additive functionals of Birkhoff sums of chaotic probability preserving dynamical systems under general assumptions.
\end{abstract}

\maketitle{}

\section*{Introduction}
Let $d\in\{1,2\}$. 
The $\mathbb Z^d$-periodic Lorentz gas coming from~\cite{Lorentz05} models the displacement of a point particle moving at unit speed in
$\mathbb R^{d}\times\mathbb T^{2-d}$ (i.e. in the plane $\mathbb R^2$ if $d=2$ and on the tube $\mathbb R\times\mathbb T$ if $d=1$) between a (non-empty) $\mathbb Z^d$-periodic configuration of obstacles, with elastic collisions off these obstacles. 
The obstacles are assumed to be open, convex, with boundary $\mathcal C^3$, non null curvature and  with pairwise disjoint closures. The number of obstacles is assumed to be locally finite. We study the $\mathbb Z^d$-periodic Lorentz gas flow $(Y_t)_{t\in\mathbb R}$ (continuous time model) via the $\mathbb Z^d$-periodic Lorentz gas map which describes the dynamics at collision times (discrete time model). 
Both for the continuous time and for the discrete time model, a state is a couple $(q,\vec v)$ made of a position 
 and a unit velocity vector.\\
 For the Lorentz gas flow $(Y_t)_{t\in\mathbb R}$, the positions are taken in the domain (outside obstacles), and to avoid any confusion, we identify, when the position is on a obstacle, pre-collisional and post-collisional vectors.  The flow $Y_t$ maps a state $(q,\vec v)$ to the state $(q_t,\vec v_t)$ at time $t$ of a point particle that was in state  $(q,\vec v)$ at time $0$. This map preserves the Lebesgue measure $\mathfrak m$.\\
 For the Lorentz gas map $T$, the states are the couples of a position on the boundary of an obstacle and of a unit post-collisional vector. 
We write $M$ for the set of such states and we consider the collision map $T:M\rightarrow M$
which maps a state at a collision time to the state at the next reflection time. This maps $T$ preserves a measure $\mu$, invariant by translation, equivalent to the Lebesgue measure and normalized so that $\mu\left(M\cap ([0;1[^2\times\mathbb S^1)\right)=1$. This measure $\mu$ is infinite.\\
The horizon is said to be finite if there exist no line touching no obstacle, then the horizon is actually bounded, meaning that the distance of a trajectory between two collisions is uniformly bounded. It is said to be infinite otherwise.

Some results are now knwon to be true both in finite and infinite horizon. 
For $d\in\{1,2\}$, the $\mathbb Z^d$-periodic Lorentz gas is known to be 
recurrent (both in finite horizon \cite{Conze99,SV04} and in infinite horizon \cite{SV07}) and ergodic (\cite{Simanyi89,P00}).
A crucial point in the study of this model is that it is a $\mathbb Z^d$-extension of the Sinai billiard map $(\bar M,\bar T,\bar \mu)$, which enjoys a chaotic behaviour. This probability measure preserving system corresponding to the Lorentz gas modulo $\mathbb Z^d$ for the position is mixing \cite{Sinai70}, enjoys exponential decorrelation for H\"older observables (both in finite horizon\cite{Young98} and in infinite horizon \cite{Chernov99} and enjoys a standard Central Limit Theorem \cite{Chernov99} for H\"older observables (both in finite horizon \cite{BS81,BCS91} and in infinite horizon \cite{Chernov99}).

When the horizon is finite, the position of a particle after $n$ collisions satisfies a standard Central Limit Theorem \cite{BS81,BCS91} (due to
a general argument \cite{Zweimuller07}, this result holds true with respect to any probability measure absolutely continuous with respect to the Lebesgue measure). Among the results proved for the Lorentz gas in finite horizon, let us mention mixing rate in infinite measure \cite{SV04,P18,P19,DNP22} including expansions of any order (both for the map \cite{P19} and for the flow \cite{DNP22}), limit theorems for Birkhoff sums (both for integrable observables \cite{DSV1}, and for smooth integrable observables with null integral \cite{PTho1,PTho2}), quantitative recurrence estimates (estimate for the tail probability of the first return time in the initial cell \cite{DSV1}, limit theorem for the return time in a neighbourhood of the initial state or of its initial position \cite{PS10}), limit theorem for the self-intersections number \cite{P14,Phalempin1}, study of differential equations perturbed by the Lorentz gas \cite{Phalempin_thesis}, etc.

In the present article, we focus on the case when the horizon is fully-dimensional infinite, i.e. when there exist $d$ non parallel
infinite lines touching no obstacle. In this case, the time between two consecutive collisions is not bounded anymore, and even worth it is not square integrable with respect to the invariant probability measure $\bar\mu$
of the Sinai billiard map. It is still possible to apply operator techniques as in the finite horizon case, but with a loss of important nice properties, and the study requires much more delicate study.
Roughly speaking, whereas the Fourier-perturbed operators family $t\mapsto P_t$ is $\mathcal C^\infty$ from $\mathbb R$ to $\mathcal L(\mathcal B)$ for some nice Banach space $\mathcal B$ when the horizon is finite, when the horizon is infinite $t\mapsto P_t$ is only smooth (and less than $\mathcal C^2$) from
$\mathbb R$ to $\mathcal L(\mathcal B\rightarrow L^1)$, this complicates seriously the work with iterates or expansion this operator. Nevertheless some results have been established in this infinite horizon context, overcoming these difficulties by creative ideas combined with additional technicality. A first specific result is the  non-standard Central Limit Theorem satisfied by the position at the $n$-th collision time~\cite{SV07}. This result was established together with a non-standard local limit theorem for the cell label~\cite{SV07}, leading to a mixing rate in $(n\log n)^{-\frac d2}$ for the Lorentz gas map. Whereas a mixing expansion has been established in the finite horizon case, this does not seem reachable in the infinite horizon case because of the weak smoothness properties of $t\mapsto P_t$. Nevertheless, further mixing estimates, including an error term and also different mixing rates for some specific null integral smooth observables have been established in~\cite{PTerhesiu1}. These shows the variety of different possible mixing rates in the infinite horizon case. Among the recent results in infinite horizon, let us mention
an estimate on the tail probability of the first return time of the map $T$ to the initial cell~\cite{PTerhesiu1}, a Local Large Deviation (LLD) estimate~\cite{MelbournePeneTerhesiu}, a mixing rate in $(t\log t)^{-\frac d2}$ for the flow~\cite{PTerhesiu2} for natural observables (such as indicator functions of balls). The proof of this last result required a coupled version of the above mentioned LLD, combined with several new tricks such as a large deviation estimate on the time of the $n$-th collision, a joint mixing local limit theorem, a new tightness-type criteria, etc. 

The goal of the present article is to investigate the behaviour of
ergodic sums of the Lorentz gas map $T$, that is of 
partial sums of the form $\sum_{k=0}^{n-1}f\circ T^k$ for integrable observables, including the case of null integral observables.
We restrict our study to the case where the observables only depend on the cell label. This restriction allows us to treat natural interesting cases such as the number of visits to the 0-cell, or the difference between the number of visits to two different cells. 
Unfortunately the general study of \cite{PTho1} does not apply to this context because of the lack of smoothness of $t\rightarrow P_t\in\mathcal L(\mathcal B)$. Nevertheless, we found a way to implement the moment method used in \cite{PTho1} despite these difficulties and to establish a limit theorem for some Birkhoff sums of null integral observables of the $\mathbb Z^d$-periodic Lorentz gas (Theorem~\ref{main1}). The observables we consider are functions of the local time in the Lorentz-gas cells. In particular it applies to the difference between the number of collisions (among the $n$ first collisions) in the cell labeled by $a\in\mathbb Z^d$ and the number of collisions in another cell labeled by $b\in\mathbb Z^d$ with $b\ne a$. 
Setting ${\mathfrak a}_k:=\max(1,\sqrt{k\log k})$, when the observable is integrable, the Birkhoff sum will be normalized by ${\mathfrak A}_n:=\sum_{k=1}^n {\mathfrak a}_k^{-d}$, whereas it will be normalized
by $\sqrt{\mathfrak A_n}$ when the observable has null integral (and satisfies some integrability assumption). 
We reinforce Theorem~\ref{main1} in a joint functional limit theorem
(Theorem~\ref{main2}) valid for the Birkhoff sum of a couple
$(g,f)$ with $g$ integrable and $f$ having a null integral, and obtain as a consequence an analogous result for the Lorentz gas flow (Theorem~\ref{mainflow}). Observe that 
\[
\mathfrak A_n\sim 2\sqrt{\frac n{\log n}}\quad\mbox{ if }\quad d=1\, ,
\quad\mbox{and}\quad 
\mathfrak A_n\sim\log\log n\quad\mbox{ if }\quad d=2\, .
\]
The present article is organized as follows. In Section~\ref{Sec1} we present our main results for the periodic Lorentz gas flow in infinite horizon. These results will appear as an application of general results stated in a general framework in Section~\ref{secgene}
completed with Appendix~\ref{Append}.
In Section~\ref{verif}, we present a general strategy to prove our general assumptions of Section~\ref{secgene} via Fourier type operator perturbation techniques and we use this approach to prove our main results stated in Section~\ref{Sec1}.

\section{Main results for the periodic Lorentz gas in infinite horizon}\label{Sec1}
\subsection{Limit theorem for Birkhoff sums for the map}
Let us start by introducing some additional notations. 
The obstacles are given by $\mathcal O_i+\ell$ with $\ell\in\mathbb Z^d$ and with $i=1,...,I$ for some $I\in\mathbb N^*$ (up to identifying $\mathbb Z^1$ with $\mathbb Z\times\{0\}$ when $d=1$). We write $\mathcal C_\ell$ (and call it $\ell$-cell) for the set of states $(q,\vec v)\in M$ based on $\bigcup_{i=1}^I\mathcal O_i+\ell$.
We identify $\bar M$ with $\mathcal C_0$ and the Sinai billiard map $\bar T$ to the map $\bar T:\bar M\rightarrow\bar M$ corresponding to the quotient map of $T$ modulo $\mathbb Z^d$ for positions. 
Let $\bar \Psi:\bar M\rightarrow\mathbb Z^d$ be the cell change function, i.e., for all $\bar x\in\bar M=\mathcal C_0$, 
$T(\bar x)\in \mathcal C_{\bar \Psi(\bar x)}$.
This can be rewritten as follows
\[
\forall \bar x=(q,\vec v)\in\bar M=\mathcal C_0,\quad \bar T(\bar x)=(q',\vec v')\ \Rightarrow \ T(\bar x)=\left(q'+\bar\Psi(\bar x),\vec v'\right) \, .
\]
More generally, by $\mathbb Z^d$-periodicity,
 \[
 \forall \bar x=(q,\vec v)\in\bar M=\mathcal C_0,\ \forall \ell\in\mathbb Z^d,\quad \bar T(\bar x)=(q',\vec v')\ \Rightarrow \ T((q+\ell,\vec v))=\left(q'+\ell+\bar\Psi(\bar x),\vec v'\right) \, .
 \]
It then follows by a direct induction that
 \[
\forall \bar x=(q,\vec v)\in\bar M=\mathcal C_0,\ \forall \ell\in\mathbb Z^d,\quad \bar T^n(\bar x)=(q_n,\vec v_n)\ \Rightarrow \ T^n((q+\ell,\vec v))=\left(q_n+\ell+\bar S_n(\bar x),\vec v_n\right) \, ,
\]
where $\bar S_n:=\sum_{k=0}^{n-1}\bar\psi\circ \bar T^k$. 
This means that the dynamics of the Lorentz gas is totally determined by the joint dynamics of the Sinai billiard and of the Birkhoff sum $\bar S_n$. 
In other words, $(M,T,\mu)$
can be represented as the $\mathbb Z^d$-extension of $(\bar M,\bar T,\bar\mu)$ by $\bar\Psi$. 
Recall that Sz\'asz and Varj\'u proved in~\cite{SV07} that $(\bar S_n/\mathfrak a_n)_n$ converges in distribution to a Gaussian distribution. Let us write $\Phi$ for the density function of this Gaussian
random variable.

\begin{thm}\label{main1}

Let $f$ be a $\mu$-integrable function constant on the cells $\mathcal C_\ell$. Then
$
\left(\sum_{k=0}^{n-1}f\circ T^k / {\mathfrak A}_n\right)_n
$
converges in distribution
(with respect to any probability measure absolutely continuous with respect to the Lebesgue measure on $M$)
to $\Phi(0)\int_Mf\, d\mu\, |\mathcal Z|$, where
\begin{itemize}
\item $\mathcal Z$ is a standard gaussian distribution if $d=1$,
\item $\mathcal Z$ is a random variable with standard exponential distribution if $d=2$.
\end{itemize}
If furthermore $\int_Mf\, d\mu=0$ and $\int_{M}(1+d(0,\cdot))^{\frac{2+\varepsilon-d}2}|f| \, d\mu<\infty$ for some $\varepsilon\in(0,1/2)$, then
$
\left(\sum_{k=0}^{n-1}f\circ T^k /\sqrt{{\mathfrak A}_n}\right)_n
$
converges in distribution (with respect to any probability measure absolutely continuous with respect to the Lebesgue measure on $M$) to $\sqrt{\sigma_f^2\Phi(0)|\mathcal Z|}\mathcal N$
where $\mathcal N$ is a random variable, independent of $\mathcal Z$, with standard gaussian and where $\sigma_f^2=\sum_{k\in\mathbb Z}\int_M f.f\circ T^k\, d\mu$.
\end{thm}
The first part of Theorem~\ref{main1} is a direct consequence of~\cite{SV07} via moment estimates used e.g. in~\cite{DSV1} for the periodic Lorentz gas
with finite horizon. Even in the case of the finite horizon, the second part of Theorem~\ref{main1} (study of Birkhoff sums of null integral) is more delicate to establish since it requires a delicate care of cancellations in order to identify the main order terms in compositions of perturbed operators using their expansions (see~\cite{PTho1}). When the horizon is infinite, taking care of these cancellations become even more challenging since the perturbed operators do not admit expansion as a family of operators, but only expansions as a family of linear maps from $\mathcal B$ to $L^1$, forbidding direct compositions of these expansions. This additional difficulty comes from the fact that the cell change function $\bar \Psi$ is {\bf not square integrable} with respect to $\bar\mu$ (whereas it is bounded and so finite-valued when the horizon of the Lorentz process is finite).
\begin{rmk}\label{LLN}
It follows from the first part of Theorem~\ref{main1} with the notations therein combined with the Hopf ergodic theorem\footnote{The Hopf ergodic theorem states that, since $(M,\mu,T)$ is recurrent ergodic, the sequence of ergodic ratios $\left(\sum_{k=0}^{n-1}h\circ T^k/\sum_{k=0}^{n-1}\mathbf 1_{\mathcal C_0}\circ T^k\right)_{n\ge 1}$ converges almost surely to $\int_Mh\, d\mu/\mu(\mathcal C_0)=\int_Mh\, d\mu$.} that for any 
$\mu$-integrable $h:M\rightarrow \mathbb R$, 
$
\left(\sum_{k=0}^{n-1}h\circ T^k / {\mathfrak A}_n\right)_n
$
converges in distribution to $\Phi(0)\int_Mh\, d\mu\, |\mathcal Z|$.
This result can be seen as a weak Law of Large Number for the infinite measure preserving dynamical system $(M,T,\mu)$.
\end{rmk}

\subsection{Functional limit theorem for the map and for the flow}
\begin{thm}\label{main2}
Let  $f,g:M\rightarrow\mathbb R$ be two integrable functions, with $f$ constant on each cell $\mathcal C_\ell$.
	Assume	furthermore that $\int_{M}(1+d(0,\cdot))^{\frac{2+\varepsilon-d}2}|f| \, d\mu<\infty$ for some $\varepsilon\in(0,\frac 12)$ and $\int_Mf\, d\mu=0$, then
	the following family of couple of processes
	\begin{equation}\label{jointprocessbill}
	\left(\left(\sum_{k=0}^{\lfloor nt\rfloor-1}g\circ T^k / {\mathfrak A}_n\right)_t,\left(\sum_{k=0}^{\lfloor nt\rfloor-1}f\circ T^k /\sqrt{{\mathfrak A}_ n}\right)_t\right)_{n\ge 1}
	\end{equation}
	converges in distribution (with respect to any probability measure absolutely continuous with respect to the Lebesgue measure on $M$) to
	$((\int_Mg\, d\mu\mathcal L_t)_t,(B_{\sigma_f^2\mathcal L_t})_t)$
	(in $(\mathcal D([0;T]))^2$ for all $T>0$ if $d=1$ and in $(\mathcal D([T_0;T]))^2$ for all $0<T_0<T$ if $d=2$), 
		where $\sigma^2_f$ is the quantity introduced in Theorem~\ref{main1}, 
	where $ B$ is a standard brownian motion independent of the process $\mathcal L_t$ where
	\begin{itemize}
		\item if $d=1$, $\mathcal L_t$ is the local time at 0 in the time interval $[0;t]$ of the Brownian motion $ W$ limit in distribution of $(\bar S_{\lfloor nt\rfloor}/\sqrt{n})_t$ as $n\rightarrow +\infty$,
		\item if $d=2$, for all $t>0$ $\mathcal L_t=\mathcal L_1$ 
		is a random variable with exponential distribution with mean $\Phi(0)$.
	\end{itemize}.
\end{thm}
Our proofs of Theorems~\ref{main1} and~\ref{main2} are given in Section~\ref{verif}. They rely on the general results (in a general framework) stated in the Section~\ref{secgene}. 

As an immediate consequence of Theorem~\ref{main2} combined with the classical random time change result (see e.g.~\cite[Chapter 14]{Bill}), we obtain the following result valid for the periodic Lorentz gas flow. Let us write $\mathcal M$ for the set of states of the 
Lorentz gas flow, and $\mathfrak m$ for the Lebesgue measure on $\mathcal M$, and $\mathcal N_t(\ell)$ for the number of collisions of the flow in the cell
$\mathcal C_\ell$ up to time $t$.
\begin{thm}\label{mainflow}
For any real valued sequence $(\beta_\ell)_{\ell\in\mathbb Z^d}$ such that $\sum_{\ell\in\mathbb Z^d}(1+|\ell|)^{\frac{2+\varepsilon-d}2}|\beta_\ell|$ for some $\varepsilon>0$ and $\sum_{\ell\in\mathbb Z^d}\beta_\ell=0$, then the family of processes
	\begin{equation}\label{jointprocesses2a}
	\left(\left(\mathfrak A_n^{-1}\mathcal N_{nt}(0),\mathfrak A_{n}^{-\frac 12}
	\sum_{\ell\in\mathbb Z^d}\beta_\ell\mathcal N_{nt}(\ell)\right)_t\right)_{n\ge 1}
	\end{equation}
	converges in distribution (with respect to any probability measure absolutely continuous with respect to the Lebesgue measure on $\mathcal M$) to
	$((
	\mathcal L'_t)_t,(B_{\sigma^2\mathcal L'_t})_t)$
	(in $(\mathcal D([0;T]))^2$ for all $T>0$ if $d=1$ and in $(\mathcal D([T_0;T]))^2$ for all $0<T_0<T$ if $d=2$), 
	where $\sigma^2$ is quantity $\sigma^2_f$ introduced in Theorem~\ref{main1}
	for the function $f$ such that $f_{|\mathcal C_\ell}\equiv \beta_\ell$	 for all $\ell\in\mathbb Z^d$, or equivalently, with
	\begin{align*}
	\sigma^2&:=\sum_{k\in\mathbb Z}\sum_{a,b\in\mathbb Z^d} \beta_a\beta_b\bar\mu\left(\bar S_{k}=b-a \right)\\
	&=\sum_{k\in\mathbb Z}\sum_{a,b\in\mathbb Z^d} \beta_a\beta_b\left(\bar\mu\left(\bar S_{k}=b-a \right)-\bar\mu\left(\bar S_{k}=b\right)-\bar\mu\left(\bar S_{k}=-a\right)+\bar\mu\left(\bar S_{k}=0\right)\right)\, .
	\end{align*}
	where $ B$ is a standard brownian motion independent of the process $\tilde{\mathcal L}_t$ where
	\begin{itemize}
		\item if $d=1$, $\mathcal L'_t$ is the local time at 0 in the time interval $[0;t]$ of the Brownian motion $W'$ limit in distribution, as $n\rightarrow +\infty$, of $((q_{nt}/\sqrt{n})_t)_n$ where $q_{nt}$ is the first coordinate of the position of $Y_{nt}(\cdot)$, i.e. the first coordinate of the position at time $nt$ of the particle,
		\item if $d=2$, for all $t>0$ $\mathcal L'_t=\mathcal L'_1$ 
		is a random variable with exponential distribution with mean $\Phi(0)$.
	\end{itemize}. 
\end{thm}

\begin{proof}[Proof of Theorem~\ref{mainflow}]
	Recall that the Sinai billiard  flow (at unit speed) endowed with the Lebesgue measure $\mathfrak m$ can naturally be represented by the suspension flow over $(\bar M,\bar T,2\sum_{i=1}^I|\partial O_i|\bar \mu)$ with roof function
	$\tau$, the time before the next collision.
	Indeed, this representation consists in identifying each $y\in\mathcal M$ with the unique couple $(x,s)$ such that $x\in M$,  $s\in[0;\tau(x))$ and $y=Y_s(x)$ ($x$ corresponds to the state at the previous collision time and $s$ to the time spent since this previous colliqion time).  
	Since the Sinai billiard system is ergodic, it follows from the Birkhoff ergodic theorem that the number of collisions $\mathfrak n_{nt}$ for the billiard flow (or equivalently for the $\mathbb Z^d$-periodic Lorentz gas flow) in the time interval $[0;nt]$ is almost surely equivalent to $nt/\mathbb E_{\bar\mu}[\tau]$ as $n\rightarrow +\infty$. 
	We conclude by applying Theorem~\ref{main2} to $g=
	\mathbf 1_{\mathcal C_0}
	$ and to $f=\sum_{\ell\in\mathbb Z^d}\beta_\ell\mathbf 1_{\mathcal C_\ell}$, and using the random time $\mathfrak n_{nt}$, that
	\[
	\left(\left(\mathfrak A_{\mathfrak n_{n}}^{-1}
	\mathcal N_{nt}(0)
	,\mathfrak A_{\mathfrak n_{n}}^{-\frac 12}
	\sum_{\ell\in\mathbb Z^d}\beta_\ell\mathcal N_{nt}(\ell)\right)_t\right)_{n}
	\]
	has the same limit in distribution $(X^{(1)},X^{(2)})$  as~\eqref{jointprocessbill} for the above choice of $(g,f)$. 
	Now, using the fact that, almost surely, as $t\rightarrow +\infty$, $\mathfrak A_{\mathfrak n_{n}}\sim \mathfrak A_n/(\mathbb E_{\bar\mu}[\tau])^{\frac {2-d}2}$, 
	we conclude that the joint process~\eqref{jointprocesses2a} converges
	in distribution to the joint process $((\mathbb E_{\bar\mu}[\tau])^{\frac {2-d}2}X^{(1)},(\mathbb E_{\bar\mu}[\tau])^{\frac {2-d}4}X^{(2)})$.
	But, by the same random time change argument, 
	the Brownian motion $W'$ 
	corresponds to the Brownian motion limit of $((\bar S_{\mathfrak n_{nt}}/\sqrt{n})_t)_n$, i.e. to the Brownian motion limit of $(\bar S_{nt/\mathbb E_{\bar\mu}[\tau]}/\sqrt{n})_n$, i.e. to $W/\sqrt{\mathbb E_{\bar\mu}[\tau]}$. Thus the local time $(\mathcal L'_t)_t$ of $W'$ at 0, is equal
	to  $(\sqrt{\mathbb E_{\bar\mu}[\tau]}\mathcal L_t)_t$, recalling that $\mathcal L$ is the local time  of $W$ at 0. This ends the proof of the corollary.
\end{proof}

\begin{rmk}
	Let $G:\mathcal M\rightarrow \mathbb R$ be an integrable function with respect to the Lebesgue measure 
	$\mathfrak m$ on $\mathcal M$ (velocity vectors $\vec v\in\mathbb S^1$ being identified with an angle in $\mathbb R/\mathbb Z$). 
	The proof of Theorem~\ref{main2} can be adapted, by taking 
defined on $M$ by $g(x):=\int_0^{\tau(x)}G(Y_s(x))\, ds$
and by noticing that
\[
\int_{M}g\, d\mu=\frac{\int_{\mathcal M}G\, d\mathfrak m}{2\sum_{i=1}^I|\partial O_i|}
\]
to prove that	 the family of processes
	\begin{equation*}\label{jointprocesses2}
	\left(\left(\mathfrak A_n^{-1}\int_0^{nt}G\circ Y_s\, ds,\mathfrak A_n^{-\frac 12}
	\sum_{\ell\in\mathbb Z^d}\beta_\ell\mathcal N_{nt}(\ell)\right)_{t}\right)_{n}
	\end{equation*}
	converges in distribution (with respect to any probability measure absolutely continuous with respect to the Lebesgue measure on $Q\times\mathbb S^1$) to
	$\left(\left(\frac 1{2\sum_{i=1}^I|\partial O_i|}
	\int_MG\, d\mathfrak m\, 
	\mathcal L'_t\left)_t,\right(B_{\sigma^2\mathcal L'_t}\right)_t\right)$ in the same sense as in Theorem~\ref{mainflow},
	with $\sigma^2$ the quantity appearing in Theorem~\ref{mainflow}.
	
\end{rmk}

\section{General results}\label{secgene}
Observe that the Birkhoff sums considered in Theorem~\ref{main1} can be
rewritten as additive functionals $(\sum_{k=0}^{n-1}\beta_{\bar S_k})_n$
of the Birkhoff sums $(\bar S_n)_n$ with respect to the Sinai billiard system $(\bar M,\bar T,\bar \mu)$. We keep this formulation in the present
section and state limit theorems for additive functionals of 
Birkhoff sums of a probability preserving dynamical systems under
general assumptions expressed in terms of operators. We will
see in Section~\ref{verif} how these assumptions can be proved using Fourier-perturbations of the transfer operator and how this result can be used to prove Theorem~\ref{main1}.
\subsection{General assumptions}
\begin{hypo}\label{HHH}
	Let $d\in\{1,2\}$ and $\alpha\in[d;2]$.
	Let $(\Delta,F,\nu)$ be a probability preserving dynamical system with transfer operator $P$. Let $\Psi:\Delta\rightarrow \mathbb Z^d$. 
	For any $a\in\mathbb Z^d$ and any non-negative integer $n$, we set $S_n:=\sum_{k=0}^{n-1}\Psi\circ F^k$ and we set $Q_{n,a}$ for the operator given by
	\[Q_{n,a}:=P^n(\mathbf 1_{\{S_n=a\}}\cdot)\, .\]
	There exist a $(1/\alpha)$-regularly varying sequence $(\mathfrak a_n)_{n\ge 0}$ 
	such that $\mathfrak A_n:=\sum_{k=0}^{n}\mathfrak a_n^{-d}\rightarrow +\infty$ as $n\rightarrow +\infty$ and
	a Banach space $(\mathcal B,\Vert\cdot\Vert_{\mathcal B})$ preserved by the operators $Q_{n,a}$  such that
	\begin{equation}\label{HypBanach}
	\mathbf 1_\Delta\in\mathcal B\hookrightarrow L^1(\nu)\, , 
	\end{equation}
	where the notation $\hookrightarrow$ means a continuous inclusion,
	we assume furthermore that 
	\begin{equation}\label{normQk0}
	\Vert Q_{n,0}\Vert_{\mathcal B}=\mathcal O(\mathfrak a_n^{-d})
	\end{equation}
	and that there exists $\Phi(0)>0$ such that
	\footnote{$\Phi(0)$ will appear to be the value at $0$ of the density function $\Phi$ of the limit in distribution of $(S_n/\mathfrak a_n)_n$ (see Section~\ref{verif}).}
	\begin{equation}\label{Qk0}
	Q_{n,0}
		= \Phi(0){\mathfrak a}_n^{-d}\mathbb E_{\nu}[\cdot]+  o({\mathfrak a}_n^{-d})
	\quad\mbox{in }\mathcal L(\mathcal B\rightarrow
	L^{1}(\nu))\, .	\end{equation}
\end{hypo}
In~\cite{PTho1}, to study Birkhoff sums of the periodic Lorentz gas with finite horizon, we used the following condition
\[
	Q_{n,a}
= \Phi(a/\mathfrak a_n){\mathfrak a}_n^{-d}\mathbb E_{\nu}[\cdot]+  o({\mathfrak a}_n^{-d})
\quad\mbox{in }\mathcal L(\mathcal B)\, .
\]
A crucial difference between this condition and the assumptions of the present article is that~\eqref{Qk0} is much weaker since it holds in  $\mathcal L(\mathcal B\rightarrow
L^{1}(\nu))$ instead of $\mathcal L(\mathcal B)$. In practice, this weaker condition comes
from the fact that the family of perturbed operators $t\mapsto P_t\in\mathcal L(\mathcal B)$ behind (see Section~\ref{verif}) is not continuous, but that 
$t\mapsto P_t\in\mathcal L(\mathcal B\rightarrow \mathcal L^1(\nu))$
is continuous.\\
To study additive functionals $\sum_{k=0}^{n-1}\beta_{S_k}$
with $\sum_{a\in\mathbb Z^d}\beta_a=0$, we will reinforce the previous assumption as follows.
\begin{hypo}\label{HHH1}
Assume Hypothesis~\ref{HHH} and that
\begin{equation}\label{normQka}
\sup_{a\in\mathbb Z^d}\left\Vert Q_{k,a}\right\Vert_{\mathcal B}=\mathcal O(\mathfrak a_n^{-d}) \, ,
\end{equation}
and that, for all $\eta\in[0;1]$,
\begin{equation}
\left\Vert 
Q'_{k,a,b}:=Q_{k,b}-Q_{k,a}\right\Vert_{\mathcal B} = \mathcal O\left(|b-a|^{\eta}{\mathfrak a}_k^{-d-\eta}\right)\, ,\label{Q'majo1}
\end{equation}
and
\begin{equation}
\left\Vert Q''_{k,a,b}:=
Q_{k,b-a}-Q_{k,b}-Q_{k,-a}+Q_{k,0}\right\Vert_{\mathcal B} = \mathcal O\left((|a||b|)^{\eta}{\mathfrak a}_k^{-d-2\eta}\right)\, ,\label{Q'majo}
\end{equation}
uniformly in $a,b\in\mathbb Z^d$.
\end{hypo}
\subsection{Limit theorem for additive functionals of Birkhoff sums}
\begin{thm}\label{MAIN}
	Assume Hypothesis~\ref{HHH}. 
	Then
	$
	\left(\sum_{k=0}^{n-1}\mathbf 1_{\{S_k=0\}} / {\mathfrak A}_n\right)_{n\ge 1}
	$
	converges in distribution (and in the sense of moments), with respect to $\nu$, to $\Phi(0)
	 \mathcal Y$, where $\mathcal Y$ is a Mittag-Leffler distribution of index $\frac{\alpha-d}\alpha$, i.e.
	\[\mathbb E[\mathcal Y^N]:=N!\frac{\Gamma(1+\frac{\alpha-d}\alpha)^N}{\Gamma(1+N\frac{\alpha-d}\alpha)}\, .
	\]
	If furthermore $\sum_{\ell\in\mathbb Z^d}|1+|\ell||^{\eta}|\beta_\ell|<\infty$
	with $\eta:=\frac{\alpha+\varepsilon-d}2$ for some $\varepsilon\in(0,1/2)$ and 
	$\sum_{\ell\in\mathbb Z^d}\beta_\ell=0$, and if Hypothesis~\ref{HHH1} holds true, then
	$
	\left(\sum_{k=0}^{n-1}\beta_{S_k} /\sqrt{{\mathfrak A}_n}\right)_{n\ge 1}
	$
	converges in distribution (and in the sense of moments), with respect to $\nu$, to $\sqrt{\sigma_\beta^2\Phi(0)\mathcal Y}\mathcal N$
	where $\mathcal N$ is a random variable, independent of $\mathcal Z$, with standard gaussian and where 
	\begin{align}\label{sigma21}
	\sigma_\beta^2&:=\sum_{k\in\mathbb Z}\sum_{a,b\in\mathbb Z^d} \beta_a\beta_b\nu\left(S_{|k|}=b-a \right)\\
	&=\sum_{k\in\mathbb Z}\sum_{a,b\in\mathbb Z^d} \beta_a\beta_b\left(\nu\left(S_{|k|}=b-a \right)-\nu\left(S_{|k|}=b\right)-\nu\left(S_{|k|}=-a\right)+\nu\left(S_{|k|}=0\right)\right)\, .\label{sigma22}
	\end{align}
\end{thm}
\begin{rmk}
The summability assumption of $\beta_\ell$ appearing in Theorem~\ref{MAIN} is to our knowledge the optimal one even in the case of additive observables of random walks with i.i.d. increments.
\end{rmk}
\begin{rmk}
It follows from our assumptions that, if $\beta$ is not identically null, 
only the second sum~\eqref{sigma22} defining $\sigma^2_\beta$ is absolutely convergent in $k,a,b$. Indeed \[\nu(S_{k}=b-a)-\nu(S_{k}=b)-\nu(S_k=-a)+\nu(S_k=0)=
	\mathbb E_\nu[Q''_{k,a,b}(\mathbf 1)]
	\]
 is summable in $(k,a,b)\in\mathbb N\times\mathbb Z^d\times\mathbb Z^d$, whereas $\nu(S_{k}=0)=\mathbb E_{\nu}[Q_{k,0}(\mathbf 1)]\sim\Phi(0)\mathfrak a_k^{-d}$ is not summable. The summability of~\eqref{sigma22}
	combined with the fact that
\[\forall k\ge 0,\quad \sum_{a,b\in\mathbb Z^d}\beta_a\beta_b\nu(S_k=b-a)=\sum_{a,b\in\mathbb Z^d}\beta_a\beta_b\mathbb E_\nu[Q''_{k,a,b}(\mathbf 1)]\]
implies the absolute convergence in $k$ of the sum appearing in the right hand side of~\eqref{sigma21}.
\end{rmk}

\begin{proof}[Proof of Theorem~\ref{MAIN}]
	Let $(\beta_\ell)_{\ell\in\mathbb Z^d}$ be as in the assumption of
	Theorem~\ref{MAIN}.  
We start by writing 
\begin{align}
\nonumber
\mathbb E_{\nu}&\left[\left(\sum_{k=0}^{n-1}\beta_{S_k}\right)^N\right]
=
\sum_{k_1,...,k_N=0}^{n-1}\mathbb E_{\nu}\left[\prod_{j=1}^N\beta_{ S_{k_j}}\right]\\
&\quad =\sum_{0\le k_1\le ...\le k_N\le n-1}c_{(k_1,...,k_N)}\sum_{a_1,...,a_N\in\mathbb Z^d}\mathbb E_\nu\left[\prod_{j=1} ^N\left(\beta_{a_j}\mathbf 1_{\{S_{k_j}=a_j\}}
\right)\right]\, ,
\label{momentN}
\end{align}
where we denote by $c_{(k_1,...,k_N)}$ the number of $N$-uples $(k'_1,...,k'_N)\in\{0,...,n-1\}^N$
such that there exists a permutation $\sigma\in\mathfrak S_N$
such that  $k'_i=k_{\sigma(i)}$ for all $i=1,...,N$. We observe that
\begin{align}
\mathbb E_\nu\left[\prod_{j=1} ^N\left(\beta_{a_j}\mathbf 1_{\{S_{k_j}=a_j\}}
\right)\right]
\nonumber& = 
\mathbb E_\nu\left[\prod_{j=1} ^N\left(\beta_{a_j}\mathbf 1_{\{S_{k_j-k_{j-1}}=a_j-a_{j-1}\}}\circ F^{k_{j-1}}
\right)\right]\\
& =\mathbb E_{\nu}\left[P^{k_N}\left(\prod_{j=1}^N\left(\beta_{a_j}\mathbf 1_{ \{S_{k_j-k_{j-1}}=a_j-a_{j-1}\}}\circ F^{k_{j-1}}\right)\right)\right]\, ,\label{muSk}
\end{align}
setting $a_0:=0$ and $P$ for the transfer operator of $F$ with respect to $\nu$, using the fact that $\mathbb E_\nu[\cdot]=\mathbb E_\nu[P^{k_N}(\cdot)]$.
Since $P^k(f.g\circ F^k)=gP^k(f)$, we observe that, 
for any $j=1,...,N$, for any $k_1\le...\le k_j$ and any $b_1,...,b_j\in\mathbb Z^d$, 
\begin{align}\label{Pprod}
P^{k_j}\left(
\prod_{i=1}^j\mathbf 1_{\{S_{k_i-k_{i-1}}=b_i\}}\circ F^{k_{i-1}}\right)=
P^{k_{j}-k_{j-1}}\left(\mathbf 1_{\{S_{k_j-k_{j-1}}=b_j\}}P^{k_{j-1}}
\left(
\prod_{i=1}^{j-1}\left(
\mathbf 1_{\{S_{k_i-k_{i-1}}=b_i\}}\circ F^{k_{i-1}}\right)\right)
	\right)\, .
\end{align}
It follows from~\eqref{muSk} and~\eqref{Pprod} that
\begin{align}\label{nuprod}
\mathbb E_\nu\left[\prod_{j=1} ^N\left(\beta_{a_j}\mathbf 1_{\{S_{k_j}=a_j\}}
\right)\right]
=\mathbb E_{\nu}\left[ \beta_{a_N}Q_{k_{N}-k_{N-1},a_N-a_{N-1}}(\cdots (\beta_{a_2}Q_{k_2-k_1,a_2-a_1}(\beta_{a_1}Q_{k_1,a_1}(\mathbf 1)))\cdots)  \right]\, .
\end{align}
For the first part of Theorem~\ref{MAIN}, we apply~\eqref{momentN} with  $\beta_\ell=\mathbf 1_{\{\ell=0\}}$ with~\eqref{nuprod}. In that case,  
applying repeatedly~\eqref{Qk0}, combined with~\eqref{normQk0} and~\eqref{HypBanach}, the right hand side of~\eqref{nuprod} becomes
\[
\mathbb E_{\nu}\left[ Q_{k_N-k_{N-1},0}(
...(Q_{k_1,0}(\mathbf 1)))
\right]=\left(\Phi(0)\right)^N \prod_{j=1}^N{\mathfrak a}_{k_j-k_{j-1}}^{-d}+\sum_{j=1}^No({\mathfrak a}_{k_j-k_{j-1}}^{-d}))\prod_{i\ne j}\mathcal O(a_{k_i-k_{i-1}}^{-d})
\]
and so that
\begin{align*}
\mathbb E_{\nu}\left[\left(\sum_{k=0}^{n-1}\mathbf 1_{\{S_k=0\}}\right)^N\right]&= o(\mathfrak A_n^{N})+N!\left(\Phi(0)\int_Mf\, d\mu\right)^N \sum_{1< k_1< ...< k_N\le n-1}\prod_{j=1}^N {\mathfrak a}_{k_j-k_{j-1}}^{-d}\\
&=o(\mathfrak A_n^{N})+ N! 
\mathfrak A_n^N\left[\left(\Phi(0)\int_Mf\, d\mu\right)^N  \frac{\Gamma(1+\frac{2-d}2)^N}{\Gamma(1+N\frac{2-d}2)}
+o(1)\right]\, ,
\end{align*}
using~\cite[Lemma 2.7]{PTho1} for the last estimate. This  ends the proof of the first part of Theorem~\ref{MAIN}. 

Now let us prove the second part. We assume from now on that
 $\sum_{a\in\mathbb Z^d}\beta_a=0$ and that
 $\sum_{a\in\mathbb Z^d}(1+|a|^{\eta})|\beta_a|<\infty$ with $\eta:=\frac{\alpha-d+\varepsilon}2$ for some $\varepsilon\in(0;\frac 12)$.
Recall that it follows from~\eqref{momentN} combined with~\eqref{nuprod}
that
\begin{align}\nonumber
\mathbb E_{\nu}\left[\left(\sum_{k=0}^{n-1}\beta_{S_k}\right)^N\right]
&
=\sum_{0\le k_1\le ...\le k_N\le n-1}c_{(k_1,...,k_N)}\sum_{a_1,...,a_N\in\mathbb Z^d}
\\
&\mathbb E_{\nu}\left[ \beta_{a_N}Q_{k_{N}-k_{N-1},a_N-a_{N-1}}(\cdots
(\beta_{a_2}Q_{k_2-k_1,a_2-a_1}(\beta_{a_1}Q_{k_1,a_1}(\mathbf 1)))\cdots)  \right]\, .\label{momentN'}
\end{align}
In Formula~\eqref{momentN'}, we decompose each $Q_{k,a}$ in 
$Q^{(0)}_{k,a}
+Q^{(1)}_{k,a}$, with
$Q_{k,a}^{(0)}:=Q_{k,0}$ and $Q_{k,a}^{(1)}:=Q'_{k,a}:=Q'_{k,0,a}=Q_{k,a}-Q_{k,0}$.
Thus
\begin{align}
&\mathbb E_{\nu}\left[\left(\sum_{k=0}^{n-1}\beta_{S_k}\right)^N\right]  = 
\sum_{\varepsilon_1,...,\varepsilon_N\in\{0,1\}} H_{\varepsilon_1,...,\varepsilon_N}^{(n,N)}\, ,
\label{momentNbis}
\end{align}
with

\begin{align*}
H_{\varepsilon_1,...,\varepsilon_N}^{(n,N)}
&
=\sum_{0\le k_1\le ...\le k_N\le n-1}c_{(k_1,...,k_N)}H_{\mathbf k,\boldsymbol{\varepsilon}}^{n,N}
\end{align*}
setting $\mathbf k=(k_1,...,k_N)$,  $\boldsymbol{\varepsilon}=(\varepsilon_1,...,\varepsilon_N)$ and
\begin{align*}
&H_{\mathbf k,\boldsymbol{\varepsilon}}^{n,N}:=
\sum_{a_1,...,a_N\in\mathbb Z^d}
\mathbb E_{\nu}\left[ \beta_{a_N}Q^{(\varepsilon_N)}_{k_{N}-k_{N-1},a_N-a_{N-1}}(\cdots
(\beta_{a_2}Q_{k_2-k_1,a_2-a_1}^{(\varepsilon_1)}(\beta_{a_1}Q_{k_1,a_1}(\mathbf 1)))\cdots)  \right]\\
&= 
\sum_{a_N\in\mathbb Z^d}
\mathbb E_{\nu}\left[\beta_{a_N}\left( \sum_{a_{N-1}\in\mathbb Z^d}\beta_{a_{N-1}}Q_{k_{N}-k_{N-1},a_N-a_{N-1}}^{(\varepsilon_N)}\left(\cdots \left(\sum_{a_1\in\mathbb Z^d}\beta_{a_1}Q_{k_{2}-k_{1},a_2-a_{1}}^{(\varepsilon_2)}\left(Q_{k_{1},a_1}^{(\varepsilon_1)}(\mathbf 1)\right)\right)\cdots\right) \right) \right]\, .
\end{align*}
We observe that $H^{(n,N)}_{\varepsilon_1,...,\varepsilon_N}=0$ as soon as 
$\varepsilon_N=0$ or if there exists $j_0=1,...,N-1$ such that $\varepsilon_{j_0}=\varepsilon_{j_0+1}=0$
(since $\sum_{a_{j_0}\in\mathbb Z^d}\beta_{a_{j_0}}=0$).
Thus, we restrict our study to the case of the $\varepsilon_j's$ 
for which the $j_0$'s such that $\varepsilon_{j_0}=0$ are isolated and
do not include $N$. 
Let such an $\boldsymbol{\varepsilon}=(\varepsilon_1,...,\varepsilon_N)$. 
Observe that there are at most $N/2$ indices $j$'s such that
$\varepsilon_j=0$. 
We set, by convention, $\varepsilon_{N+1}=\varepsilon_{0}=0$. 
We will prove that
\begin{equation}\label{dominH}
H_{\varepsilon_1,...,\varepsilon_N}^{n,N}=
 o\left(\mathfrak A_n^{\frac N2}\right)\quad\mbox{unless if }\#\{j:\varepsilon_j=0\}=\frac N2\, ,
\end{equation}
i.e. $H_{\varepsilon_1,...,\varepsilon_N}^{n,N}=
o\left(\mathfrak A_n^{\frac N2}\right)$ unless if $N$ is even and if $(\varepsilon_1,...,\varepsilon_N)=(0,1,...,0,1)$. 
For reader's convenience, we first give a short proof of this estimate in a particular case. A proof of this estimate in the general case $\sum_{a\in\mathbb Z^d} |a|^{\frac{\alpha-d+\varepsilon}2}|\beta_a|<\infty$ with  $\eta=\frac{\alpha-d+\varepsilon}2$, with $\varepsilon\in(0;\frac 12)$ is given in Appendix~\ref{Append}. 
\begin{itemize}
\item \underline{A short proof of~\eqref{dominH} in a particular case}.
Assume in this item only that $\alpha<d+1$ and 
$\sum_{a\in\mathbb Z^d}(1+|a|)^{2\bar\eta}|\beta_a|<\infty$ with $\alpha-d<\bar\eta<1$.
It follows from Hypothesis~\ref{HHH1} that 
\begin{align*}
Q^{(\varepsilon_j)}_{k_{j}-k_{j-1},a_{j}-a_{j-1}}
	&= \mathcal O\left(
|a_j-a_{j-1}|^{\varepsilon_j\bar\eta} \mathfrak a_{k_{j}-k_{j-1}}^{-d-\varepsilon_j\bar\eta}\right)\\
	&= \mathcal O\left(
((1+|a_j|)(1+|a_{j-1}|))^{\varepsilon_j\bar\eta} \mathfrak a_{k_{j}-k_{j-1}}^{-d-\varepsilon_j\bar\eta}\right)
\, ,
\end{align*}
and so that
\begin{align*}
\left|H_{\varepsilon_1,...,\varepsilon_N}^{n,N}\right|
&=\mathcal O\left(\left(\sum_{a\in\mathbb Z^d}(1+|a|)^{2\bar\eta}|\beta_a|\right)^N\prod_{j=1}^N
\sum_{k_j=0}^{n-1}\mathfrak a_k^{-d-\varepsilon_j\bar\eta}\right)\\
&=\mathcal O\left(\mathfrak A_n^{\#\{j:\varepsilon_j=0\}}
\right)\, ,
\end{align*}
where we used the fact that $\sum_{a\in\mathbb Z^d}(1+|a|)^{2\bar\eta}|\beta_a|<\infty$ combined with 
$\sum_{k\ge 0}\mathfrak a_k^{-d-\bar\eta}<+\infty$, since $d+\bar\eta>\alpha$ and since $(\mathfrak a_n)_n$  is $\frac 1 {\alpha}$-regularly varying. This concludes the proof of~\eqref{dominH}.
\end{itemize}

In particular, Formula~\eqref{dominH} ensures that 
\begin{equation}\label{momentimpair}
\mathbb E_{\nu}\left[\left(\sum_{k=0}^{n-1}\beta_{S_k}\right)^N\right]=o\left(\mathfrak A_n^{\frac N2}\right)\quad\mbox{if }N\mbox{ is odd}\, ,
\end{equation}
and that
\begin{equation}\label{momentpair}
\mathbb E_{\nu}\left[\left(\sum_{k=0}^{n-1}\beta_{S_k}\right)^N\right]=H^{(n,N)}_{0,1,...,0,1}+o\left(\mathfrak A_n^{\frac N2}\right)\quad\mbox{if }N\mbox{ is even}\, .
\end{equation}
Assume from now on that $N$ is even, then 
\begin{align}
\nonumber&H^{(n,N)}_{0,1,...,0,1}=\sum_{0\le k_1\le ...\le k_N\le n-1}c_{k_1,...,k_N}\sum_{a_1,...,a_N\in\mathbb Z^d}\left(\prod_{j=1} ^N\beta_{a_j}\right)
\\
\nonumber&
\mathbb E_{\nu}[Q'_{k_N-k_{N-1},a_N-a_{N-1}}(Q_{k_{N-1}-k_{N-2},0}
(...(Q'_{k_2-k_1,a_2-a_1}(Q_{k_1,0}(\mathbf 1)))...))]
\end{align}
Hence
\begin{align}
H^{(n,N)}_{0,1,...,0,1}&= \sum_{0\le k_1\le ...\le k_N\le n-1}\!\!\!\!\! c_{k_1,...,k_N}
\mathbb E_{\nu}[\bar Q_{k_N-k_{N-1}}(Q_{k_{N-1}-k_{N-2},0}
(...(\bar Q_{k_2-k_1}(Q_{k_1,0}(\mathbf 1)))...))] \, ,
\label{H101010}
\end{align}
with
\begin{align}\label{barQk1}
\bar Q_k&:=\sum_{a,b\in\mathbb Z^d}\beta_a\beta_bQ'_{k,b-a}=\sum_{a,b\in\mathbb Z^d}\beta_a\beta_bQ_{k,b-a}
=\sum_{a,b\in\mathbb Z^d}\beta_a\beta_b Q''_{k,a,b}\, , 
\end{align}
where we used the fact that $\sum_{a\in\mathbb Z^d}\beta_a=0$ and the notation $Q''_{k,a,b}=Q_{k,b-a}-Q_{k,b}-Q_{k,-a}+Q_{k,0}$ introduced in~\eqref{Q'majo}. 
Combining~\eqref{Q'majo} with $\eta=\frac{\alpha+\varepsilon-d}2$ with~\eqref{barQk1}, since $\sum_{a\in\mathbb Z^d}(1+|a|)^\eta|\beta_a|<\infty$, we infer that
\begin{equation}
\bar Q_k=\mathcal O\left({\mathfrak a}_k^{-\alpha-\varepsilon}\right)\quad\mbox{in }\mathcal L(\mathcal B)
\, ,\label{barQk}
\end{equation}
since $d+2\eta=\alpha+\varepsilon$, which ensures the summability of $\Vert \bar Q_k\Vert_{\mathcal B}$. 
The study of~\eqref{H101010} leads us to the question of estimating $\mathbb E_\nu[\bar Q_{k'}(Q_{k,0}(h))]$. Unfortunately we cannot compose  directly~\eqref{Q'majo} and~\eqref{Qk0}
since this last estimate is in $\mathcal L(\mathcal B\rightarrow L^1(\nu))$
and not in $\mathcal L(\mathcal B)$. But, proceeding in two steps, 
we will prove that
\begin{equation}\label{ESTI0}
\mathbb E_{\nu}[\bar Q_{k'}(Q_{k,0}(h))]-\Phi(0)\mathfrak a_k^{-d}\mathbb E_{\nu}\left[\bar Q_{k'}(\mathbf 1)\right]\mathbb E_\nu[h]
= \mathcal O\left(
\mathfrak a_{k'}^{-\alpha-\varepsilon'}\Vert h \Vert_{\mathcal B}\right)o(\mathfrak a_k^{-d})\, ,
\end{equation}
where $\varepsilon'\in(0,\frac 12)$ is small enough so that $\frac{(\alpha-d+2\varepsilon')(\alpha+\varepsilon')}{\alpha+2\varepsilon'}\le \alpha-d+\varepsilon=2\eta$.

First, dominating separately both terms, it follows from~\eqref{Q'majo}
with $\eta'=\frac{\alpha+2\varepsilon'-d}2\in(0,1)$ and from~\eqref{HypBanach} and~\eqref{normQk0} that
\begin{align}
&\mathbb E_{\nu}[Q''_{k',a,b}(Q_{k,0}(h))]-\Phi(0)\mathfrak a_k^{-d}\mathbb E_{\nu}\left[Q''_{k',a,b}(\mathbf 1)\right]\mathbb E_\nu[h]
=\mathcal O\left(|a|^{\eta'}|b|^{\eta'}\mathfrak a_{k'}^{-\alpha-2\varepsilon'}\mathfrak a_k^{-d}\Vert h\Vert_{\mathcal B}\right)\, ,\label{ESTI1}
\end{align}
since $-d-2\eta'=-\alpha-2\varepsilon'$. 
Second, it follows from the definition of $Q''_{k',a,b}$
that
\[
\mathbb E_\nu[Q''_{k',a,b}(h_0)]=\mathbb E_\nu\left[(\mathbf 1_{\{S_{k'}=b-a\}}-\mathbf 1_{\{S_{k'}=b\}}-\mathbf 1_{\{S_{k'}=-a\}}+\mathbf 1_{\{S_{k'}=0\}}).h_0\right]
=\mathcal O\left(\Vert h_0\Vert_{L^1(\nu)}\right)\, .
\]
This combined with~\eqref{Qk0} ensures that
\begin{align}
\nonumber&
\mathbb E_{\nu}[Q''_{k',a,b}(Q_{k,0}(h))]-\Phi(0)\mathfrak a_k^{-d}\mathbb E_{\nu}\left[Q''_{k',a,b}(\mathbf 1)\right]\mathbb E_\nu[h]\\
&\quad=\mathbb E_{\nu}\left[Q''_{k',a,b}(Q_{k,0}(h)-\Phi(0)\mathfrak a_k^{-d}\mathbb E_\nu[h])\right]
=\mathcal O(
\Vert h\Vert_{\mathcal B})o(\mathfrak a_k^{-d})\, .\label{ESTI2}
\end{align}
Thus, combining~\eqref{ESTI1} and~\eqref{ESTI2}, we obtain that
\begin{align*}
&
\mathbb E_{\nu}[Q''_{k',a,b}(Q_{k,0}(h))]-\Phi(0)\mathfrak a_k^{-d}\mathbb E_{\nu}\left[Q''_{k',a,b}(\mathbf 1)\right]\mathbb E_\nu[h]\\
&\quad= \left(\mathcal O(|a|^{\eta'}|b|^{\eta'}\mathfrak a_{k'}^{-\alpha-2\varepsilon'}\mathfrak a_{k}^{-d}\Vert h\Vert_{\mathcal B})\right)^{(\alpha+\varepsilon')/(\alpha+2\varepsilon')}\left(\mathcal O(
\Vert h\Vert_{\mathcal B})o(\mathfrak a_k^{-d})\right)^{\varepsilon'/(\alpha+2\varepsilon')}\\
&\quad= \mathcal O(|a|^{\frac {\eta'(\alpha+\varepsilon')}{\alpha+2\varepsilon'}}|b|^{\frac {\eta'(\alpha+\varepsilon)}{\alpha+2\varepsilon'}}\mathfrak a_{k'}^{-\alpha-\varepsilon'}\Vert h \Vert_{\mathcal B})o(\mathfrak a_k^{-d})\, .
\end{align*}
After summation over $a,b\in\mathbb Z^d$, we obtain~\eqref{ESTI0}, since $\frac {\eta'(\alpha+\varepsilon')}{\alpha+2\varepsilon'}\le\eta$ and since 
$\sum_{a\in\mathbb Z^d}(1+|a|^\eta) |\beta_a|<\infty$.

Using~\eqref{ESTI0}  inductively  in~\eqref{H101010} (combined with the fact that $\Vert Q_{k,0}\Vert_{\mathcal B}=\mathcal O(\mathfrak a_k^{-d})$ and that $\Vert\bar Q_{k}\Vert_{\mathcal B}$ is summable which follows from~\eqref{barQk}), we conclude that, when $N$ is even
\begin{align*}
H^{(n,N)}_{0,1,...,0,1}
&=\sum_{0\le k_1\le...\le k_N}c_{k_1,...,k_N}
\prod_{j=1}^{N/2}\left(\Phi(0) \left({\mathfrak a}_{k_{2j-1}-k_{2j-2}}^{-d}\mathbb E_{\nu}[\bar Q_{k_{2j}-k_{2j-1}}(\mathbf 1)]\right)\right)\\
&\quad +\mathcal O\left(  \left(\sum_{k=1}^n{\mathfrak a}_k^{-d}
\right)^{N/2-1}
\sum_{k=1}^n
o({\mathfrak a}_{k
}^{-d
})\right)\, .
\end{align*}
Recall that 
\[\sigma_\beta^2=\sum_{k\in\mathbb Z}\mathbb E_{\nu}\left[\bar Q_{|k|}(\mathbf 1)\right] =\mathbb E_\nu\left[\bar Q_0(\mathbf 1)\right]+2\sum_{k\ge 1}\mathbb E_{\nu}\left[\bar Q_{k}(\mathbf 1)\right]\, .
\]
Therefore, proceeding exactly as in~\cite[p. 1918-1919]{PTho1}, we obtain that
\begin{align*}
H^{(n,N)}_{0,1,...,0,1}
&=
\frac{\Gamma(1+\frac{\alpha-d}\alpha)^{N/2}}{\Gamma(1+\frac N2\frac{\alpha-d}\alpha)}\frac{N!}{2^{N/2}}
\left(\Phi(0)\sigma^2_\beta
\mathfrak A_n\right)^{N/2}+o\left(\mathfrak A_n^{ N/2}\right)\\
&= \mathfrak A_n^{N/2}
\mathbb E\left[ \left(\sqrt{\sigma_\beta^2\Phi(0)|\mathcal Z|}\mathcal N\right)^N\right]+o\left(\mathfrak A_n^{ N/2}\right)\, .
\end{align*}
This, combined with~\eqref{momentimpair} and~\eqref{momentpair},
ends the proof of the convergence of every moments. We conclude the convergence in distribution by the Carleman criterion~\cite{Schmudgen}.
\end{proof}

\subsection{Joint Limit theorem for additive functional of Birkhoff sums}
\begin{thm}\label{MAIN2}
Assume Assumptions~\ref{HHH} and~\ref{HHH1}.  	
Let	$\eta:=\frac{\alpha+\varepsilon-d}2$ for some $\varepsilon\in(0,1/2)$. 
Let $(\beta_a^{(0)})_{a\in\mathbb Z^d}$ and $(\beta_a^{(1)})_{a\in\mathbb Z^d}$ be two families of real numbers such that 
$\sum_{a\in\mathbb Z^d}(1+|a|)^\eta|\beta^{(j)}_a|<\infty$ and 
$\sum_{a\in\mathbb Z^d}\beta^{(1)}_a=0$. Then
	the following family of couples of processes
	\begin{equation}\label{jointprocess}
	\left(\left(\sum_{k=0}^{\lfloor nt\rfloor-1}\beta_{S_k}^{(0)} / {\mathfrak A}_n\right)_t,\left(\sum_{k=0}^{\lfloor nt\rfloor-1}\beta^{(1)}_{S_k} /\sqrt{{\mathfrak A}_ n}\right)_t\right)_{n\ge 1}
	\end{equation}
	converges in distribution, with respect to $\nu$, to
	$((\sum_{a\in\mathbb Z^d}\beta^{(0)}_a\mathcal L_t)_t,(\sigma_{\beta^{(1)}}B_{\mathcal L_t})_t)$,
	(in $(\mathcal D([0;T]))^2$ for all $T>0$ if $d=1$ and in $(\mathcal D([T_0;T]))^2$ for all $0<T_0<T$ if $d=2$), 
	where $\sigma^2_{\beta^{(1)}}$ is defined in Formula~\eqref{sigma21} of
	Theorem~\ref{MAIN} taking $\beta=\beta^{(1)}$, 
	where $ B$ is a Brownian motion and where $\mathcal L_t$ is the following process  
		\begin{itemize}
		\item if $\alpha>d$, $\mathcal L_t$ is the local time at 0 in the time interval $[0;t]$ of an centered $\alpha$-stable process $W$, independent of $B$, such that $W_1$ has density probability $\Phi$ with $\Phi(0)$ satisfying~\eqref{Qk0},
		\item if $\alpha=d$, $\mathcal L_t=\mathbf 1_{\{t>0\}}\mathcal L_1$, where $\mathcal L_1$
		is a random variable with exponential distribution with mean $\Phi(0)$.
	\end{itemize}
\end{thm}
\begin{proof}[Proof of Theorem~\ref{MAIN2}]
	We start by proving the convergence of the finite distributions and we will then prove the tightness.
	For the convergence of the finite distributions, we use again the convergence of moments.
	It is enough to study the asymptotic behaviour as $n$ goes to infinity
	of every moments of the following form
	\[
	E_n:=\mathbb E_{\nu}\left[
	\prod_{j=1}^M\left(\sum_{k_j=\lfloor n t_{j-1}\rfloor}^{\lfloor n t_j\rfloor-1}\beta^{(0)}_{S_k}\right)^{N^{(0)}_j}\left(\sum_{k'_j=\lfloor n t_{j-1}\rfloor}^{\lfloor n t_j\rfloor-1}\beta_{S_k}^{(1)}\right)^{N^{(1)}_j}\right]\, ,
	\]
	for any $M\in\mathbb N^*$, any $N_j^{(0)},N_j^{(1)}\in\mathbb N$, any $t_0=0<t_1<...<t_M$. We set $\Gamma_k:=\sum_{j=1}^{M}N^{(k)}_j$ for $k\in\{0,1\}$ and will prove that
	\begin{align*}
	\lim_{n\rightarrow +\infty}\mathfrak A_n^{-\Gamma_0-\frac{\Gamma_1}2}E_n&=(\Phi(0))^{\Gamma_0+\frac{\Gamma_1}2}\left(
	\sum_{a\in\mathbb Z^d}\beta^{(0)}_a\right)^{\Gamma_0}\sigma_{\beta^{(1)}}^{\Gamma_1} \mathbb E\left[
	\prod_{j=1}^M \left(\mathcal L_{t_j}-\mathcal L_{t_{j-1}}\right)^{N^{(0)}_j+\frac{N^{(1)}_j}2}
	\right]\prod_{j=1}^M\mathbb E[\mathcal N^{N^{(1)}_j}]\, .
	\end{align*}
	We set $m_i:=\sum_{r=1}^{i}(N^{(0)}_r+N^{(1)}_r)$ and $N:=m_M$. 
	Proceeding as in the proof of Theorem~\ref{main1}, we observe that $E_n$ can be rewritten as follows
	\begin{align}
	\sum_{\gamma_1,...,\gamma_{N}}\sum_{(k_1, ..., k_{N})\in \mathcal K_n}&\prod_{i=1}^{N}c_{(k_{m_{i-1}+1}
		,...,k_{m_i})}^{(\gamma_{m_{i-1}+1}
		,...,\gamma_{m_i})}\sum_{a_1,...,a_{N}\in\mathbb Z^d}\left(\prod_{j=1} ^{N}\beta_{a_j}^{(\gamma_j)}\right)\nu\left(\forall i=1,...,N,\, S_{k_i}=a_i\right)\, ,
	\label{momentN2}
	\end{align}
	where $\mathcal K_n$ is the set of increasing $N$-uples $(k_1,...,k_{N})$ such that $\lfloor nt_{j-1}\rfloor\le k_m\le
	\lfloor nt_{j}\rfloor -1$ for all $m$ such that $m_{j-1}<m\le m_j$, where the first sum is taken over $\gamma_j\in\{0,1\}$ such that, for all $\gamma\in\{0,1\}$ and all $i=1,...,M$,
	$\#\{j=m_{i-1}+1,...,m_i\, :\, \gamma_j=\gamma\}=N^{(\gamma)}_i$ and
	where $c_{(k_{m_{i-1}+1}
		,...,k_{m_i})}^{(\gamma_{m_{i-1}+1}
		,...,\gamma_{m_i})}$ is the number of $(\gamma'_{m_{i-1}+1},
	,...,\gamma'_{m_i},k'_{m_{i-1}+1}
	,...,k'_{m_i})\in\{0,1\}^{N^{(1)}_i+N^{(2)}_i}\times\{\lfloor nt_{i-1}\rfloor,...,\lfloor nt_i\rfloor-1\}^{N_i^{(0)}+N_i^{(1)}}$
	such that there exists a permutation $\sigma$ of $\{m_{i-1}+1,...,m_i\}$
	such that $\gamma'_{\sigma(r)}=\gamma_r$ and $k'_{\sigma(r)}=k_r$ for all $r\in\{m_{i-1}+1,...,m_i\}$.
	Furthermore  we use~\eqref{nuprod} to express $\nu(\forall i=1,...j-1,N,\ {k_i}=a_i)$ using a composition of operators $Q_{k_j-k_{j-1},a_j-a_{j-1}}$ and, as in Appendix~\ref{Append}, for each $j$, we decompose each $Q_{k_j-k_{j-1},a_j-a_{j-1}}$ in a sum of $\widetilde Q^{(\varepsilon)}$. 
	This leads to
	\begin{align}
	E_n:=\sum_{\boldsymbol\gamma=(\gamma_1,...,\gamma_{N})}\sum_{\boldsymbol\varepsilon=(\varepsilon_1,...,\varepsilon_{N})} \tilde H^{(n)}_{\boldsymbol\varepsilon
	}(\boldsymbol\gamma
)\, ,
	\label{momentNbis2}
	\end{align}
	with
	\begin{align*}
	&\tilde H^{(n)}_{\varepsilon_1,...,\varepsilon_{N}}(\gamma_1,...,\gamma_{N})
	=\sum_{(k_1, ..., k_{N})\in \mathcal K_n}
	\prod_{i=1}^Mc_{(k_{m_{i-1}+1}
		,...,k_{m_i})}^{(h_{m_{i-1}+1}
		,...,h_{m_i})}
	\sum_{a_1,...,a_{N}\in\mathbb Z^d}\\
	&
	\mathbb E_{\nu}\left[\beta_{a_N}^{(\gamma_N)}\left( 
	\beta_{a_{N-1}}^{(\gamma_{N-1})}\widetilde Q_{k_{N}-k_{N-1},a_{N-1},a_N}^{(\varepsilon_N)}\left(\cdots \left(
	\beta^{(\gamma_1)}_{a_1}\widetilde Q_{k_{2}-k_{1},a_{1},a_2}^{(\varepsilon_2)}\left(\widetilde Q_{k_{1},0,a_1}^{(\varepsilon_1)}(\mathbf 1)\right)\right)\cdots\right) \right) \right]\, ,
	\end{align*}
	with the use of the operators $\widetilde Q_{(k,a,b)}^{(\varepsilon)}$
	defined in Appendix~\ref{Append} and where the sum over $\varepsilon_1,...,\varepsilon_N$ is taken over $\varepsilon_2,...,\varepsilon_N\in\{0,1\}^2$ and  $\varepsilon_1\in\{(0,0),(0,1)\}$. We write $\varepsilon_j=(\varepsilon_{j,1},\varepsilon_{j,2})$. 
	Since $\sum_{a\in\mathbb Z^d}\beta_a^{(1)}=0$, 
	$H_{\varepsilon_1,...,\varepsilon_N}=0$ if there exists $j=1,...,N$
	such that $\gamma_j=1$ and $\varepsilon_{j,2}+\varepsilon_{j+1,1}=0$ (with convention $\varepsilon_{N+1,1}=0$). 
Therefore we assume from now on that $(\boldsymbol\gamma, \boldsymbol\varepsilon)$ is such that $\varepsilon_{j,2}+\varepsilon_{j+1,1}\ge \gamma_j$, for all $j$
(with the convention $\varepsilon_{N+1,1}=0$, and with $\varepsilon_{1,1}=0$) and we call {\bf admissible} such a pair $(\boldsymbol\gamma,\boldsymbol\varepsilon)$. 
Then
\begin{align*}
\tilde H^{(n)}_{\boldsymbol\varepsilon}(\boldsymbol\gamma)
&=\mathcal O\left(\sum_{(k_1, ..., k_{N})}\sum_{a_1,...,a_{N}\in\mathbb Z^d}
\prod_{j=1}^{N}\left(|\beta^{(\gamma_j)}_{a_j}|\left\Vert   \widetilde Q^{(\varepsilon_j)}_{k_j-k_{j-1},a_{j-1},a_j}\right\Vert_{\mathcal B}\right)\right)\\
&=\mathcal O\left(\sum_{(k_1, ..., k_{N})}\sum_{a_1,...,a_{N}\in\mathbb Z^d}
\prod_{j=1}^{N}\left(|\beta^{(\gamma_j)}_{a_j}|\, (1+|a_j|)^{\eta
} \mathfrak a_{k_j-k_{j-1}}^{-d-\eta_{j,1}\varepsilon_{j,1}-\eta_{j,2}\varepsilon_{j,2}}\right)\right)\, ,
\end{align*}
with $(\eta_{j,i})_{j,i}$ a sequence in $\{0,\eta\}$ such that
$\eta_{j,2}\varepsilon_{j,2}+\eta_{j+1,1}\varepsilon_{j+1,1}=\eta\gamma_j$
and such that
\begin{itemize}
	\item $(\eta_{j,1},\eta_{j,2})=\eta \varepsilon_j$ if $\varepsilon_j\ne (1,1)$,
	\item $\eta_{j,1}=\eta(1-\varepsilon_{j-1,2})$ if $\varepsilon_j=(1,1)$,
	\item $\eta_{j,2}=\eta$ if $\varepsilon_j=(1,1)$ and $\varepsilon_{j+1}\ne (1,0)$,
	\item $\eta_{j,2}=0$ if $\varepsilon_j=(1,1)$ and $\varepsilon_{j+1}= (1,0)$.
\end{itemize}
As seen in Appendix~\ref{Append}, we use the fact that there exists $u_0\in(0,1]$ such that $\sum_{k=0}^{n-1}\mathfrak a_k^{-d-\eta}=\mathcal O\left(\mathfrak A_n^{\frac {1-u_0}2}\right)$ and that $\sum_{k\ge 0}\mathfrak a_k^{-d-2\eta}<\infty$. Using the summability assumption on $\beta_a^{(\gamma)}$, we infer that
\begin{align*}
\tilde H^{(n)}_{\boldsymbol\varepsilon}(\boldsymbol\gamma)
&=\mathcal O\left(\sum_{(k_1, ..., k_{N})}
\prod_{j:\varepsilon_{j,1}+\varepsilon_{j,2}=0}\mathfrak a_{k_j-k_{j-1}}^{-d-\eta_{j,1}\varepsilon_{j,1}-\eta_{j,2}\varepsilon_{j,2}}\right)=\mathcal O\left(\mathcal A_n^{\mathcal E^{\boldsymbol\eta}_{0}+\frac{1-u_0}2\mathcal E^{\boldsymbol\eta}_1}\right)\, ,
\end{align*}
where 
\[
\mathcal E_k^{\boldsymbol\eta}:=\#\{j=1,...,N:\eta_{j,1}\varepsilon_{j,1}+\eta_{j,2}\varepsilon_{j,2}=k\eta\}\, .
\]
We also set $\mathcal E_k:=\#\{j=1,...,N:\varepsilon_{j,1}+\varepsilon_{j,2}= k\}$ and observe that
$\mathcal E_1\le\mathcal E_1^{\boldsymbol\eta}$. 
Recall $\Gamma_k=\sum_{j=0}^{N}N^{(k)}_j=\#\{j:\gamma_j=k\}$.
Then
\[
N=\mathcal E_0^{\boldsymbol\eta}+\mathcal E_1^{\boldsymbol\eta}+\mathcal E_2^{\boldsymbol\eta}
=\mathcal E_0+\mathcal E_1+\mathcal E_2
=\Gamma_0+\Gamma_1\, .
\]
On the other side
\begin{align*}
\eta\Gamma_1&=\eta \sum_{j=1}^{N}\gamma_j= \sum_{j=1}^{N}(\eta_{j,2}\varepsilon_{j,2}+\eta_{j+1,1}\varepsilon_{j+1,1})\\
&=\sum_{j=1}^{N}(\eta_{j,1}\varepsilon_{j,1}+\eta_{j+1,1}\varepsilon_{j,2})-\eta_{1,1}\varepsilon_{1,1}= \eta\left(\mathcal E_1^{\boldsymbol\eta}+2\mathcal E_2^{\boldsymbol\eta}\right)\, ,\\
\end{align*}
and so
\begin{align*}
\mathcal E_0^\eta+\frac{1-u_0}2\mathcal E_1^{\boldsymbol\eta}&=
N-\mathcal E_2^{\boldsymbol\eta}-\mathcal E_1^{\boldsymbol\eta}+\frac{1-u_0}2\mathcal E_1^{\boldsymbol\eta}\\
&= \left(\Gamma_0+\Gamma_1\right)-\left(\frac{\Gamma_1-\mathcal E_1^{\boldsymbol\eta}}2
\right)-\frac{1+u_0}2\mathcal E_1^{\boldsymbol\eta}\\
&= \Gamma_0+\frac{\Gamma_1}2-
\frac{u_0}2
\mathcal E_1^{\boldsymbol\eta}\, .
\end{align*}
Hence we have proved that
\[
H^{(n)}_{\boldsymbol\varepsilon}(\boldsymbol\gamma)
=o\left(\mathcal A_n^{\Gamma_0+\frac{\Gamma_1}2}\right)\quad\mbox{if }\mathcal E_1^{\boldsymbol\eta}>0\, ,
\]
Now we assume that $\mathcal E_1^{\boldsymbol\eta}=0$, this implies that
$\mathcal E_1=0$ and so that the $j$'s such that $\varepsilon_j=(1,1)$ are isolated and $\mathcal E_2= \mathcal E_2^{\boldsymbol\eta}=\frac{\Gamma_1}2$.\\
Observe that this implies that $\Gamma_1=\sum_{j=1}^MN^{(1)}_j$ is even and
we have proved that
\begin{equation}\label{Enodd}
E_n=o\left(\mathfrak A_n^{\Gamma_0+\frac{\Gamma_1}2}\right)\quad\mbox{if }n\mbox{ is odd}\, .
\end{equation}
The above conditions on $\boldsymbol\gamma$ and $\boldsymbol\varepsilon$
imply that $\boldsymbol\varepsilon$ is a sequence of $(0,0)$ and $(1,1)$ (which are isolated) and that $\varepsilon_{j,2}+\varepsilon_{j+1,1}\ge\gamma_j$, thus 
$\varepsilon_{j,2}+\varepsilon_{j+1,1}=\gamma_j$ and
\begin{itemize}
	\item if $\gamma_j=0$, then both $\varepsilon_j=(\varepsilon_{j,1},\varepsilon_{j,2})$ and
	$\varepsilon_{j+1}=(\varepsilon_{j+1,1},\varepsilon_{j+1,2})$ are $(0,0)$.
	\item if $\gamma_j=1$, then either $\varepsilon_j=(\varepsilon_{j,1},\varepsilon_{j,2})$ or
	$\varepsilon_j=(\varepsilon_{j,1},\varepsilon_{j,2})$ is $(1,1)$, and the other one is $(0,0)$.
\end{itemize}
This means that the $j$'s such that $\gamma_j=1$ appear in pairwise disjoint pairs $(j-1,j)$ such that $(\gamma_{j-1},\gamma_{j})=(1,1)$, and that $\varepsilon_j=(1,1)$ if and only if $(j-1,j)$ is such a couple. Fixing such a pair $(\boldsymbol\gamma, \boldsymbol\varepsilon)$, let us write $\mathcal J$ for the set of such  $j$ such that $\varepsilon_j=(1,1)$. Then, using repeatedly~\eqref{Qk0} and~\eqref{ESTI1}, we obtain
	\begin{align*}
	&\mathfrak A_n^{-\Gamma_0-
		\frac{\Gamma_1}2}
	    \tilde H^{(n)}_{\boldsymbol\varepsilon}(\boldsymbol\gamma) =\sum_{(k_1, ..., k_{N})\in \mathcal K_n}
	    \prod_{i=1}^Mc_{(k_{m_{i-1}+1}
	    	,...,k_{m_i})}^{(\gamma_{m_{i-1}+1}
	    	,...,\gamma_{m_i})}
	    \prod_{j':\gamma_{j'}=0}\left( 
	    \frac{\mathfrak a_{k_{j'}-k_{j'-1}}^{-d}}{\mathfrak A_n}\Phi(0)\sum_{a\in\mathbb Z^d}\beta_a^{(0)}\right)\\
	&\quad \prod_{j\in\mathcal J}\frac {\Phi(0) \mathfrak a_{k_{j-1}-k_{j-2}}^{-d}}{2\, \mathfrak A_n}\sum_{a,b\in\mathbb Z^d}\beta^{(1)}_a\beta^{(1)}_b\mathbb E_\nu\left[Q''_{k_j-k_{j-1},a,b}(\mathbf 1)\right]+o(1)
	\end{align*}
and so
	\begin{align*}
\mathfrak A_n^{-\Gamma_0-
	\frac{\Gamma_1}2}
\tilde H^{(n)}_{\boldsymbol\varepsilon}(\boldsymbol\gamma)
	&=o\left(1\right)+\Phi(0)^{\Gamma_0+
		\frac{\Gamma_1}2}\left(\sum_{a\in\mathbb Z^d}\beta_a^{(0)}\right)^{\Gamma_0}\sum_{(k_1, ..., k_{N})\in \mathcal K_n:k_j<k_{j+1}}
	\prod_{i=1}^M(N_i^{(0)})!(N_i^{(1)})!\\
	&\quad\left(\prod_{j:\gamma_j=0}\frac{\mathfrak a_{k_j-k_{j-1}}^{-d}}{\mathfrak A_n}\right)
	\prod_{j\in\mathcal J}\frac { \mathfrak a_{k_{j-1}-k_{j-2}}^{-d}}{2\, \mathfrak A_n}
	\sum_{a,b\in\mathbb Z^d}\beta^{(1)}_a\beta^{(1)}_b\mathbb E_\nu\left[
	Q''_{k_j-k_{j-1},a,b}(\mathbf 1)\right]\, .
	\end{align*}
	It can be worthwhile to notice that 
	we can restrict the above sum on the $(k_1, ..., k_{N})\in \mathcal K_n$
	such that $k_j-k_{j-1}<\log n$ if $j\in\mathcal J$, and $k_j-k_{j-1}>\log n$ for the other values of $j$'s. This implies that
	\[
	\mathfrak A_n^{-\Gamma_0-\frac{\Gamma_1}2}\tilde H^{(n)}_{\boldsymbol\varepsilon}(\boldsymbol\gamma) 
	=o(1)
	\]
	as soon as there exist $j\in\mathcal J$ and $i\in\{1,...,M\}$ such that
	$k_{j-1}<\lfloor nt_i\rfloor\le k_{j}$ (indeed this combined with
	$k_j-k_{j-1}<\log n$ implies that $0<k_j-nt_i<\log n$ and $0<nt_i-k_{j-1}<\log n$.  
	In particular
	\[
	E_n=o\left(\mathfrak A_n^{\Gamma_0+
		\frac{\Gamma_1}2} \right)
	\quad\mbox{if }\exists j\in\{1,...,M\},\ N_j^{(1)}\in 2\mathbb Z+1\, .
	\]
	We assume from now on that the $N_j^{(1)}$'s are even and that $\mathcal J$ is such that, for every $j\in\mathcal J$, there exists $i=1,...,M$ such that  
	$k_{j-1},k_j$ are in a same set $\{\lfloor nt_{i-1}\rfloor,...,\lfloor nt_i\rfloor-1\}$. Then
	\begin{align*}
	&\mathfrak A_n^{-\Gamma_0-\frac{\Gamma_1}2}\tilde H^{(n)}_{\boldsymbol\varepsilon}(\boldsymbol\gamma) =o\left(1\right)+
	\prod_{j=1}^M\left(N_j^{(0)}!\, N_j^{(1)}!
	(\Phi(0))^{N''_j}\left( 
	\sum_{a\in\mathbb Z^d}\beta^{(0)}_a\right)^{N_j^{(0)}}\right.\\
	&
	\left(\frac 12\sum_{k\ge 0}(1+\mathbf 1_{\{k\ne 0\}})\sum_{a,b\in\mathbb Z^d}\beta_a^{(1)}\beta_b^{(1)}\mathbb E_\nu\left[\bar Q''_{k_j-k_{j-1},a,b}(\mathbf 1)\right]\right)^{\frac{N_j^{(1)}}2}
	\left.\sum_{(k'_1, ..., k'_{N''})\in\mathcal K'_n}
	\prod_{j=1}^{N''}\frac { \mathfrak a_{k'_{j}-k'_{j-1}}^{-d}}{ \mathfrak A_n}
	\right)\, ,
	\end{align*}
	with $N''_j:=N_j^{(0)}+\frac{N^{(1)}_j}2$ and $N'':=\sum_{j=1}^MN''_j$, and where
	$\mathcal K'_n$ is the set of strictly increasing sequences $k'_1\le...\le k'_{N''}$ with exactly $N''_j $ elements between $\lfloor nt_{j-1}\rfloor$
	and $\lfloor nt_j\rfloor-1$ and with the convention $k'_0=0$.
	Thus
	\begin{align*}
	&\frac{\tilde H^{(n)}_{\boldsymbol\varepsilon}(\boldsymbol\gamma)}{\mathfrak A_n^{\sum_{j=1}^MN''_j}} 
	=o\left(1\right)+E'_n
	\prod_{j=1}^M\left( N_j^{(0)}! N_j^{(1)}!
	\left(\sum_{a\in\mathbb Z^d}\beta^{(0)}_a\right)^{N_j^{(0)}}2^{-\frac {N_j^{(1)}}2}\sigma_{\beta^{(1)}}^{N_j^{(1)}}
	\frac {1}{N''_j!}\right)
	\end{align*}
	with
	\begin{align}
\nonumber	E'_n&:=(\Phi(0))^{N''}\left(\prod_{j=1}^M N''_j!\right)\sum_{(k'_1, ..., k'_{N''})\in\mathcal K'_n}
	\prod_{j=1}^{N''}\frac { \mathfrak a_{k'_{j}-k'_{j-1}}^{-d}}{ \mathfrak A_n}\\
	&=o(1)+\mathfrak A_n^{-N''}
	\mathbb E_{\nu}\left[
	\prod_{j=1}^M\left(\sum_{k_j=\lfloor n t_{j-1}\rfloor}^{\lfloor n t_j\rfloor-1}\mathbf 1_{\{S_{k_j}=0\}}\right)^{N''_j}\right]\, .\label{momentlocaltime}
	\end{align}
	We observe that there exist $\frac{N''_j!}{N_j^{(0)}! (N_j^{(1)}/2)!}$
	sequences $(\gamma_{m_j+1},...,\gamma_{m_{j+1}})\in\{0,1\}^{N^{(0)}_j+N^{(1)}_j}$ in which the $1$'s appear in $N^{(1)}_j/2$ pairwise distinct  pairs $(\gamma_{j-1},\gamma_j)$. Therefore
	\begin{align*}
	\frac{E_n}{\mathfrak A_n^{\sum_{j=1}^MN''_j}} 
	&=o\left(1\right)+ E'_n
	\prod_{j=1}^M\left( \frac{N_j^{(1)}!}{(N_j^{(1)}/2)!2^{\frac{N_j^{(1)}}2}}
	\left(\sum_{a\in\mathbb Z^d}\beta^{(0)}_a\right)^{N_j^{(0)}}\sigma_{\beta^{(1)}}^{N_j^{(1)}}\right)\\
	&=o\left(1\right)+E'_n
	\prod_{j=1}^M\left(\left(\sum_{a\in\mathbb Z^d}\beta^{(0)}_a\right)^{N_j^{(0)}}\mathbb E\left[(\sigma_{\beta^{(1)}}\mathcal N)^{N_j^{(1)}} \right]
	\right)\, .
	\end{align*}
	It remains to study the asymptotics of $E'_n$.
	\begin{itemize}
		\item If $d=1<\alpha$, we consider a  $\mathbb Z$-valued non-arithmetic random walk $(\widetilde S_n)_n$ (with i.i.d. increments) such that $(\widetilde  S_{\lfloor nt\rfloor}/\mathfrak a_n)_n$ converges in distribution to the 
		$\alpha$-stable process $W$. The previous computations hold also true  (more easily) for $\widetilde S_n$ instead of $S_n$ and lead to
		\[
		E'_n\sim \mathfrak A_n^{-N''}
		\mathbb E\left[
		\prod_{j=1}^M\left(\sum_{k_j=\lfloor n t_{j-1}\rfloor}^{\lfloor n t_j\rfloor-1}\mathbf 1_{\{\widetilde S_{k_j}=0\}}\right)^{N''_j}\right]\quad\mbox{as }n\rightarrow +\infty\, .
		\]
		But the process  $\left(\sum_{k=0}^{\lfloor nt\rfloor-1}\mathbf 1_{\{S_k=0\}}\right)_t$ of local time at 0 of $\widetilde S_n$
		converges in distribution to the process $(\mathcal L_t)_t$ of local time at 0 of $W$. This combined with the dominations of the moments of any order ensures that 
		$(E'_n)_n$ converges in distribution to
$\mathbb E\left[\prod_{j=1}^M\left(\mathcal L_{t_j}-\mathcal L_{t_{j-1}}\right)^{N''_j}\right]$. 
		We conclude that
		\begin{align*}
		&\frac{E_n}{\mathfrak A_n^{\sum_{j=1}^MN''_j}} 
		=o\left(1\right)+
		\mathbb E\left[\prod_{j=1}^M\left(\mathcal L_{t_j}-\mathcal L_{t_{j-1}}\right)^{N''_j}\right]
		\prod_{j=1}^M\left(\left(\sum_{a\in\mathbb Z^d}\beta^{(0)}_a\right)^{N_j^{(0)}}\mathbb E\left[(\sigma_{\beta^{(1)}}\mathcal N)^{N_j^{(1)}} \right]
		\right)\\
		&=o\left(1\right)+
		\mathbb E\left[\prod_{j=1}^M\left(\sum_{a\in\mathbb Z^d}\beta^{(0)}_a\left(\mathcal L_{t_j}-\mathcal L_{t_{j-1}}\right)\right)^{N_j^{(0)}}
		\prod_{j=1}^M\left(\sigma_{\beta^{(1)}}(\mathcal N_{\mathcal L_{t_j}}-\mathcal N_{\mathcal L_{t_{j-1}}})\right)^{N_j^{(1)}} \right]
		\, .
		\end{align*}

		\item If $d=\alpha$,
		then
		\begin{align*}
		\mathfrak A_n^{-N''}\mathbb E_\nu\left[\left(\sum_{k=\lfloor NT_0\rfloor}^{\lfloor Nt_M\rfloor-1}\mathfrak 1_{\{S_k=0\}}\right)^{N''}\right]
	&=\mathcal O\left(	\mathfrak A_n^{-N''}\left(\sum_{k=\lfloor NT_0\rfloor}^{\lfloor nt_M\rfloor-1}\mathfrak a_k^{-d}\right)^{N''}\right)\\
	&=\mathcal O\left(\left((\mathfrak A_{\lfloor nt_M\rfloor}-\mathfrak A_{\lfloor nT_0\rfloor})/\mathfrak A_n^{N''}\right)^{N''}\right)\, ,
		\end{align*}
converges to 0 as $n\rightarrow +\infty$, since $(\mathfrak A_n)_{n\ge 0}$ is slowly varying. Thus 
		\[
		E'_n=o(1) \quad\mbox{if }M\ge 2\, .
		\]
		Furthermore, it follows from the proof of Theorem~\ref{main1} that
		if $M=1$,
		\begin{align*}
		E'_n&=o(1)+\mathfrak A_n^{-N''}
		\mathbb E_{\nu}\left[
		\left(\sum_{k_j=\lfloor n t_{j-1}\rfloor}^{\lfloor n t_j\rfloor-1}\mathbf 1_{\{S_{k_j}=0\}}\right)^{N''_1}\right]\\
		&= (\Phi(0))^{N''_1}N''_1!=(\Phi(0))^{N''_1}\mathbb E[\mathcal E^{N''_1}]\, ,
		\end{align*}
		where $\mathcal E$ is a random variable with standard exponential distribution due to theorem~\ref{MAIN}.
		We infer that
		\begin{align*}
		&\frac{E_n}{\mathfrak A_n^{\sum_{j=1}^MN''_j}} 
		=o\left(1\right)+\mathbb E\left[(\Phi(0)\mathcal E)^{N''_1}\prod_{j=2}^M\left(\Phi(0)\mathcal E-\Phi(0)\mathcal E\right)^{N''_j}\right]
		\prod_{j=1}^M\left(\left(\sum_{a\in\mathbb Z^d}\beta^{(0)}_a\right)^{N_j^{(0)}}\mathbb E\left[(\sigma_{\beta^{(1)}}\mathcal N)^{N_j^{(1)}} \right]
		\right)\\
		&=o\left(1\right)+
		\mathbb E\left[\prod_{j=1}^M\left(\sum_{a\in\mathbb Z^d}\beta^{(0)}_a\left(\Phi(0)\mathcal E(\mathbf 1_{\{t_j>0\}}-\mathbf 1_{\{t_{j-1}>0\}})\right)\right)^{N_j^{(0)}}
		\prod_{j=1}^M(\sigma_{\beta^{(1)}}(\mathcal N_{\Phi(0)\mathcal E\mathbf 1_{\{t_j>0\}}}-\mathcal N_{\Phi(0)\mathcal E\mathbf 1_{\{t_{j}>0\}}}))^{N_j^{(1)}} \right]
		\, .
		\end{align*}
	\end{itemize}
	This combined with~\eqref{Enodd} and the Carleman's criteria~\cite{Schmudgen} ends the proof of the convergence of the finite dimensional distributions.\\
	Let us write $((X_t^{(1,n)},X_t^{(2,n)})_{t})_{n\ge 1}$ for the joint process~\eqref{jointprocess} and let us prove its tightness. 
	When $d=1<\alpha$, we set $T_0=0$, otherwise we fix some $T_0\in(0;T)$.
	We use the tightness criterion of \cite[Theorem 13.5, (13.4)]{Bill}.
	We have proved the convergence of the finite dimensional distributions.
	It remains to prove that there exist $\alpha_1>1$ and $C>0$ such that, for every $r,s,t$ such that $T_0\le r\le s\le t\le T$, for all $j\in\{1,2\},$
	\begin{equation}\label{criteretightness}
	\exists p_j\in\mathbb N^*,\quad \mathbb E_{\nu}\left[|X_t^{(j,n)}-X_s^{(j,n)}|^{p_j}|X_s^{(j,n)}-X_r^{(j,n)}|^{p_j}\right]\le C
	|t-r|^{\alpha_1}\, .
	\end{equation}
	Observe first that, if $0\le r\le s\le t\le T$ and $t-r<1/n$, then $X_t^{(n)}-X_s^{(n)}=0$
	or $X_s^{(n)}-X_r^{(n)}=0$, thus the left hand side of~\eqref{criteretightness} is null  and so~\eqref{criteretightness} holds true. 
	Assume from now on that $T_0\le r\le s\le t\le T$ and that $t-r\ge 1/n$.
	We will use the following inequality
	\[
	\mathbb E_{\nu}\left[|X_t^{(j,n)}-X_s^{(j,n)}|^{p_j}|X_s^{(n)}-X_r^{(n)}|^{p_j}\right]\le 
	\left\Vert X_t^{(j,n)}-X_s^{(j,n)}\right\Vert_{L^{2p_j}(\nu)}^{p_j}\left\Vert X_s^{(j,n)}-X_r^{(j,n)}\right\Vert_{L^{2p_j}(\nu)}^{p_j}\, .
	\]
	Thus~\eqref{criteretightness} will follow from the fact that, for any $T_0\le r<t\le T$, $|t-r|>1/n$, and $j\in\{1,2\}$, 
	\begin{equation}\label{controletightness1}
 \exists p_j\in\mathbb N^*,\quad	\sup_{a,b:r\le a<b\le t}\left\Vert X_b^{(j,n)}-X_a^{(j,n)}\right\Vert_{L^{2p_j}(\nu)}^{2p_j}\le  C|t-r|^{\alpha_1}\, .
	\end{equation}
	It follows from our previous moment computation that
	\[
	\mathbb E_{\mu_{\mathcal C_0}}\left[|X_b^{(j,n)}-X_{a}^{(j,n)}|^{2p}\right]=\mathcal O\left(
	\left(\mathfrak A_n^{-1}\sum_{k=\lfloor nr\rfloor}^{\lfloor nt\rfloor -1}
	\mathfrak a_k^{-d}\right)^{\frac{2p}j}\right)\, .
	\]
	Thus it is enough to prove that 
	\begin{equation}\label{controletightness2}
	\exists\alpha_0>0,\quad \sup_{r,t:T_0\le r<t<T,\, |t-r|\ge 1/n}
	\mathfrak A_n^{-1}\sum_{k=\lfloor nr\rfloor}^{\lfloor nt\rfloor -1}
	\mathfrak a_k^{-d}=\mathcal O\left((t-r)^{\alpha_0}\right)\, .
	\end{equation}
	Indeed, we will conclude by taking $p_j=j\lceil (2\alpha_0)^{-1}\rceil$
	so that~\eqref{controletightness1} and so~\eqref{criteretightness} hold true with $\alpha_1:=\frac{2p_j\alpha_0}j>1$.\\
Since $(\mathfrak a_n)_{n\ge 0}$ is $(1/\alpha)$-regularly varying, it follows from Karamata's theorem~\cite{Karamata,BGT} that there exist three bounded convergent sequences  $(c(n))_{n\ge 0}$ (positive, with positive limit), $(b(n))_{n\ge 0}$ (converging to $0$) and $(\theta_n)_{n\ge 0}$ 
(positive, converging to $\frac{\alpha-1}\alpha$ if $d<\alpha$ and to $0$ if $d=\alpha$, see~\cite[Proposition 1.5.9.b]{BGT})
 such that
	\[
	\forall n\ge 0,\quad na_n^{-d}=\mathfrak A_n\theta_n\quad\mbox{and}\quad \mathfrak A_n=c(n)n^{\frac{\alpha-d}\alpha}e^{\int_{1}^n\frac{b(t)}t\, dt}\, .
	\] 
Now let us choose $\alpha_0$. If $d=1<\alpha$, we set $\alpha_0:=\frac{\alpha-1}{2\alpha}$. If $d=\alpha$, we take $\alpha_0:=1$.  Up to change, if necessary, the first terms of $(b(n))_{n\ge 0}$ and $(c(n))_{n\ge 0}$, we assume without loss of generality that the sequence $(b(n))_{n\ge 0}$ is bounded by $\alpha_0$.\\ 
If $d=1<\alpha$, if $r\le 2/n$ (observe that,  when $d=\alpha$, this case does not happen for large values of $n$ since $r\ge T_0>0$) and $|t-r|\ge 1/n$, then
	\begin{equation}\label{rpetit}
	\mathfrak A_n^{-1}\sum_{k=0}^{\lfloor nt\rfloor -1}
	\mathfrak a_k^{-d}=\mathcal O\left(
	\frac{\mathfrak A_{\lfloor nt\rfloor}} {\mathfrak A_n} \right)=\mathcal O\left(t^{1-\frac 1\alpha-\alpha_0}\right)
	=\mathcal O\left(  (t-r)^{\alpha_0}\right)\, ,
	\end{equation}
	implying~\eqref{controletightness2} and so~\eqref{controletightness1}
	and~\eqref{criteretightness} in this case.\\
	
	We assume from now on that $T_0\le r<t\le T$ and $2/n<r<t\le T$ (so that $\lfloor nr\rfloor-1>0$) and $t-r\ge 1/n$ (so that $\lfloor nt\rfloor-\lfloor nr\rfloor\le 2(nt-nr)$).
	Then, using the uniform dominations on $(c(n),\theta_n,b(n))_n$
	combined with a series-integral comparison, we obtain that
	\begin{align}
	\nonumber\mathfrak A_{n}^{-1}\sum_{k=\lfloor nr\rfloor}^{\lfloor nt\rfloor -1}
	\mathfrak a_k^{-d}
	&=\sum_{k=\lfloor nr\rfloor}^{\lfloor nt\rfloor -1}\frac{\alpha-1}{\alpha k}\frac{c(k)}{c(n)}\frac{k^{\frac {\alpha-d}\alpha}}{ n^{\frac{\alpha-d}\alpha}}\theta_k
	e^{\int_{n}^k\frac{b(u)}u\, du}=\mathcal O\left(\sum_{k=\lfloor nr\rfloor}^{\lfloor nt\rfloor -1}\frac 1{ n}(k/n)^{-\frac {d}\alpha-\alpha_0}\right)\, .
   \end{align}
   \begin{itemize}
\item   If $d=1<\alpha$, it follows that
   \begin{align}
   	\nonumber\mathfrak A_{n}^{-1}\sum_{k=\lfloor nr\rfloor}^{\lfloor nt\rfloor -1}
   \mathfrak a_k^{-d}
    & =\mathcal O\left(  \left(\frac{\lfloor nt\rfloor-1}n\right)^{\alpha_0}-\left(\frac{\lfloor nr\rfloor-1}n\right)^{\alpha_0}\right)\nonumber\\
   &=\mathcal O\left(  \left(\frac{\lfloor nt\rfloor-\lfloor nr\rfloor}n\right)^{\alpha_0}\right)=\mathcal O\left(  \left(t-r\right)^{\alpha_0}\right)
    \, ,\label{rgrand}	    
    \end{align}
	ending the proof of~\eqref{controletightness2}, from which we infer~\eqref{controletightness1}
	and~\eqref{criteretightness}. This ends the proof of the tightness when $d<\alpha$.
	\item When $d=\alpha$, we obtain
	\begin{align}
	\nonumber\mathfrak A_n^{-1}\sum_{k=\lfloor nr\rfloor}^{\lfloor nt\rfloor -1}
	\mathfrak a_k^{-d}
	&=O\left(\left(\frac{\lfloor nr\rfloor}n\right)^{-\alpha_0}-\left(\frac{\lfloor{nt}\rfloor}n\right)^{-\alpha_0}\right)\\
		&=O\left(\frac{\lfloor nt\rfloor-\lfloor nr\rfloor}n\right)=\mathcal O\left(t-r\right)\, ,
	\end{align}
	from which we infer~\eqref{controletightness2},~\eqref{controletightness1},~\eqref{criteretightness}, and so the tightness in the case where $d=\alpha$.
\end{itemize}
	This ends the proof of Theorem~\ref{MAIN2}.
\end{proof}

\section{Proof of Theorems~\ref{main1} and~\ref{main2} via Fourier perturbations}\label{verif}
A strategy to prove Assumptions~\ref{HHH} and~\ref{HHH1} consists in noticing that
\begin{equation}\label{Qkaintegral}
Q_{k,a}=\frac 1{(2\pi)^d}\int_{[-\pi;\pi]^d}e^{-i\langle t,a\rangle}P_t^k(\cdot)\, dt\, ,\quad\mbox{with }P_t(h):=P(e^{i\langle t,\Psi\rangle})\, ,
\end{equation}
and to establish nice properties for $P_t$ as the one listed in the next result.
Recall that $P_t^n(h)=P^n(he^{itS_n})$.
\begin{prop}\label{criteria}
	Assume  $\mathcal B$ is a Banach space satisfying~\eqref{HypBanach} and that there exist two constants $b\in(0,\pi)$ and 	$\alpha_0>0$  such that~: 
	\begin{equation}\label{opcond1}
	\forall t\in[-b,b]^d,\ P_t^k=\lambda_t^k\Pi_t+\mathcal O(e^{-\alpha_0k})\quad \mbox{ and }\quad \sup_{b<|u|_\infty<\pi}\Vert P_u^k\Vert_{\mathcal B}=\mathcal O(e^{-\alpha_0k})\, ,
	\end{equation}
	with $(\lambda_{t/\mathfrak a_k}^k)_{k\ge 0}$ converging to the characteristic function $\varphi$ of an $\alpha$-stable distribution,
with $\Pi_t=\mathbb E_\nu[\cdot]+o(1)$ in $\mathcal L(\mathcal B\rightarrow L^1(\nu))$ as $t\rightarrow 0$, and with
	\begin{equation}\label{opcond2}
	\sup_{t\in[-b;b]^d}\Vert \Pi_t\Vert_{\mathcal B}<\infty\quad\mbox{and}\quad
	\int_{\mathbb R}(1+|t|^2)\left(\sup_{k\ge 1}|\lambda_{t/\mathfrak a_k}^k|\mathbf 1_{\{|t|<b\mathfrak a_k\}}\right)\, dt<\infty\, .
	\end{equation}
	Then Hypotheses~\ref{HHH1} (and so~\ref{HHH}) hold true with $\Phi(0)$
	the value at $0$ of the density function $\Phi$ of the 
	$\alpha$-stable distribution with characteristic function $\varphi$.
\end{prop}
\begin{proof}[Proof of Proposition~\ref{criteria}]
It follows from our assumptions that
	\begin{align}\label{controlI}
Q_{k,0}=\frac 1{(2\pi)^d}\int_{[-\pi;\pi]^d}P_t^k\, dt
&=\frac {{\mathfrak a}_k^{-d}}{(2\pi)^d}\int_{[-b {\mathfrak a}_k;b {\mathfrak a}_k]^d}\!\!\!\!\! \lambda_{t/{\mathfrak a}_k}^k\Pi_{t/{\mathfrak a}_k}\, dt+\mathcal O(e^{-\alpha_0 k})=\Phi(0)\mathfrak a_k^{-d}+o\left({\mathfrak a}_k^{-d}\right)\, ,
\end{align}
in $\mathcal L(\mathcal B\rightarrow L^1(\nu))$, 
 via the dominated convergence theorem since $\lim_{n\rightarrow +\infty }\lambda_{t/\mathfrak a_n}^n\Pi_{t/\mathfrak a_n}=\varphi(t)\mathbb E_\nu[\cdot]$ in $\mathcal L(\mathcal B\rightarrow L^1(\nu))$  and since
$\Phi(0)=\frac 1{(2\pi)^d}\int_{\mathbb R}\varphi(t)\, dt$. Thus~\eqref{Qk0} holds true.

Furthermore, for all $\eta\in[0;2]$ and $a\in\mathbb Z^d$, in $\mathcal L(\mathcal B)$,
	\begin{align}\label{controlInt}
	\int_{[-\pi;\pi]^d}|t|^\eta \left\Vert P_t^k\right\Vert_{\mathcal B}\, dt
	&={\mathfrak a}_k^{-d-\eta}\int_{[-b {\mathfrak a}_k;b {\mathfrak a}_k]^d}\!\!\!\!\! |t|^\eta\left|\lambda_{t/{\mathfrak a}_k}^k\right|\left\Vert\Pi_{t/{\mathfrak a}_k}\right\Vert_{\mathcal B}\, dt+\mathcal O\left(e^{-\alpha_0 k}\right)=\mathcal O\left({\mathfrak a}_k^{-d-\eta}\right)\, .
	\end{align}
	Using the expression~\eqref{Qkaintegral} of $Q_{k,a}$ combined with~\eqref{controlInt} with $m=0$, we obtain~\eqref{normQk0} and~\eqref{normQka}.
	Fix $\eta\in[0;1]$.
	Then
	\begin{align*}
	\nonumber Q'_{k,a,b}&=Q_{k,b}-Q_{k,a}=\frac 1{(2\pi)^d}
	\int_{[-\pi;\pi]^d}(e^{i\langle t,b\rangle}-e^{i\langle t,a\rangle})
	P_t^k(\cdot)\, dt\\
	\nonumber&=\frac 1{(2\pi)^d}
	\int_{[-\pi;\pi]^d}\mathcal O(\langle t,b-a\rangle)^\eta
	P_t^k(\cdot)\, dt\\
	&=\mathcal O\left(|b-a|^\eta{\mathfrak a}_k^{-d-\eta}\right)\quad\mbox{in }\mathcal L(\mathcal B)\, ,
	\end{align*}
	where we used again~\eqref{controlInt} combined with the bound $|e^{ix}-e^{iy}|\le\min(2,|x-y|)\le 2^{1-\eta}|x-y|^{\eta}$,
	and so we have proved~\eqref{Q'majo1}. 
	For~\eqref{Q'majo}, in the same way, we obtain
	\begin{align}
	\nonumber Q''_{k,a,b}&=Q_{k,b-a}-Q_{k,b}-Q_{k,-a}+Q_{k,0}\\
	\nonumber&=\frac 1{(2\pi)^d}
	\int_{[-\pi;\pi]^d}(e^{i\langle t,b-a\rangle}-e^{i\langle t,b\rangle}-e^{-i\langle t,a\rangle}+1)
	P_t^k(\cdot)\, dt\\
	\nonumber&=\frac 1{(2\pi)^d}
	\int_{[-\pi;\pi]^d}(e^{i\langle t,b\rangle}-1)(e^{-i\langle t,a\rangle}-1)
	P_t^k(\cdot)\, dt\\
	\nonumber&=\frac 1{(2\pi)^d}
	\int_{[-\pi;\pi]^d}\mathcal O(\langle t,b\rangle\langle t,a\rangle)^\eta
	P_t^k(\cdot)\, dt\\
	&=\mathcal O\left((|a||b|)^\eta{\mathfrak a}_k^{-d-2\eta}\right)\quad\mbox{in }\mathcal L(\mathcal B)\, ,
	\end{align}
	and so~\eqref{Q'majo}. 
\end{proof}

\begin{proof}[Proof of Theorems~\ref{main1} and~\ref{main2}]
	Let us write $\beta_\ell$ for the constant to which $f$ is equal on the $\ell$-cell $\mathcal C_\ell$. 
	The integrability assumption means that $\sum_{\ell\in\mathbb Z}|\beta_\ell|<\infty$.
	Due to \cite{Zweimuller07}, since $\mu$ is equivalent to the Lebesgue measure on $M$, it is enough to prove the results
	with respect to the measure $\mu_{\mathcal C_0}$ (the restriction of $\mu$ to $\mathcal C_0$). 
	Thus, we consider this reference measure and establish the convergence of every moment with respect to this probability measure. 
	We observe that, with the identification of $\bar M$ to $\mathcal C_0$,
	$\mu_{\mathcal C_0}$ is identified with $\bar\mu$ and  $f\circ T^k$
	is identified with $\beta_{\bar S_k}$.\\
	As in~\cite{SV07,PTerhesiu1}, we use the two Young towers~\cite{Young98,Chernov99}.
	We write $(\hat\Delta,\hat F,\hat \nu)$ for the hyperbolic tower which is an extension of $(\bar M,\bar T,\bar \mu)$, and write
	$(\Delta,F,\nu)$ for the expanding tower obtained by quotienting $(\hat\Delta,\hat F,\hat \nu)$ along stable curves.
	We write $\bar\pi:\hat\Delta\rightarrow \bar M$ and $\pi:\hat\Delta\rightarrow\Delta$ for the two measurable maps such that
	$\bar\pi_*\hat\nu=\mu$, $\pi_*\hat\nu=\nu$, $\bar T\circ\bar\pi=\bar\pi\circ\hat F$ and $F\circ\pi=\pi\circ\hat F$. 
	Since $\bar\Psi$ is constant on stable curves, there exists a function
	$\Psi:\Delta\rightarrow \mathbb Z^d$ such that $\Psi\circ\pi=\bar\Psi\circ\bar\pi$. Setting $S_n:=\sum_{k=0}^{n-1}\Psi\circ F^k$, it follows that $
	S_n\circ\pi=\bar S_n\circ\bar\pi$.
	For the first part of Theorem~\ref{main1}, as noticed in Remark~\ref{LLN},
	it is enough to prove the first convergence result for $\beta_\ell=\mathbf 1_{\{\ell=0\}}$. We will conclude by Theorem~\ref{MAIN}. 
	To prove the assumptions of Theorem~\ref{MAIN}, we show that the criterion
	given in Proposition~\ref{criteria} is satisfied here with our choice of $\mathfrak a_n$, with $\alpha=2$ and with $\Phi$ the characteristic function of the Gaussian limit distribution of $(S_n/\mathfrak a_n)_n$. 
	The assumptions of Proposition~\ref{criteria}
	have been proved in~\cite{SV07} with the use of the Banach spaces introduced in~\cite{Young98} combined with the use of the Nagaev-Guivarch perturbation method~\cite{Nag1,Guivarch,GH} via the Keller and Liverani theorem~\cite{KL} (see also~\cite{P_livre} for a general reference on this method). 
	The fact that $(\lambda_{t/\mathfrak a_k}^k)_{k\ge 1}$ converges pointwise to the characteristic function of a Gaussian random variable follows from the existence of a positive symmetric matrix $A$ such that
	$1-\lambda_t\sim \langle At,t\rangle |\log|t||$ as $t\rightarrow 0$ (this was proved in~\cite{SV07}). For the second part of~\eqref{opcond2}, one can e.g. use the fact that $|\lambda_{t/\mathfrak a_k}^k|\mathbf 1_{\{|t|<b\mathfrak a_k\}}\le e^{-c_0\min(|t|^{2-\varepsilon},|t|^{2+\varepsilon}))}$.
	Thus Proposition~\ref{criteria} holds true and Theorems~\ref{MAIN}
	and~\ref{MAIN2} apply.
	Finally, we identify the formulas of the asymptotic variances $\sigma^2_f$
	and $\sigma^2_\beta$ by noticing that
	\[
	\sum_{a,b\in\mathbb Z^2}\beta_a\beta_b\nu(S_{|k|}=b-a)
	=\int_{M}f.f\circ T^{|k|}\, d\mu= \int_{M}f.f\circ T^{k}\, d\mu\, .
	\]
	For Theorem~\ref{main2}, we deduce the result for general $g$ using ergodicity of $(M,T,\mu)$ together with the Hopf ergodic ratio theorem.
\end{proof}

\begin{appendix}
\section{Proof of~\eqref{dominH} in the general case}\label{Append}

We assume here that $\sum_{a\in\mathbb Z^d}(1+|a|^{\eta})|\beta_a|<\infty$ with $\eta:=\frac{\alpha-d+\varepsilon}2$ for some $\varepsilon\in(0;1/2)$.

In Formula~\eqref{momentN'}, we decompose  $Q_{k,a}$ using the operators $Q''_{k,a,b}$ and $Q'_{k,c}:=Q'_{k,0,c}=Q_{k,c}-Q_{k,0}$ as follows
\[
Q_{k,a}=Q''_{k,a,b}+Q_{k,b}+Q_{k,-a}-Q_{k,0}=Q''_{k,a,b}+Q'_{k,b}+Q'_{k,-a}+Q_{k,0}\, .
\]
In~\eqref{momentN'}, we replace each  $Q_{k_j-k_{j-1},a_j-a_{j-1}}$ by this decomposition, we develop and obtain
\begin{align}\label{momentN''}
\mathbb E_{\nu}\left[\left(\sum_{k=0}^{n-1}\beta_{S_k}\right)^N\right]=
\sum_{\varepsilon_1,...,\varepsilon_N
} D^{(n,N)}_{\varepsilon_1,...,\varepsilon_N}\, ,
\end{align}
summing a priori over  $(\varepsilon_1,...,\varepsilon_N)\in(\{0,1\}^2)^N$ such that $\varepsilon_1\in\{(0,0),(0,1)\}$ the following quantity
\begin{align*}
D^{(n,N)}_{\varepsilon_1,...,\varepsilon_N}
&=\sum_{0\le k_1\le ...\le k_N\le n-1}c_{(k_1,...,k_N)}\sum_{a_1,...,a_N\in\mathbb Z^d}\left(\prod_{j=1} ^N\beta_{a_j}\right)\\
&\mathbb E_\nu[\widetilde Q_{k_N-k_{N-1},a_{N-1},a_N}^{(\varepsilon_N)}...\widetilde Q_{k_2-k_{1},a_1,a_2}^{(\varepsilon_2)} \widetilde Q_{k_1,a_0,a_1}^{(\varepsilon_1)}(\mathbf 1)]\, ,
\end{align*}
with $a_0=0$ and 
where $\widetilde Q_{k,a,b}^{(0,0)}=Q_{k_j-k_{j-1},0}$, $\widetilde Q_{k,a,b}^{(1,0)}=Q'_{k,-a}$, $\widetilde Q_{k,a,b}^{(0,1)}=Q'_{k,b}$,
$\widetilde Q_{k,a,b}^{(1,1)}=Q''_{k,a,b}$.\\
 We first restrict the sum over $\varepsilon_1,...,\varepsilon_N$. 
	 We observe that, since $\sum_{a_j\in\mathbb Z^d}\beta_{a_j}=0$, $D_{\varepsilon_1,...,\varepsilon_N}^{n,N}=0$
	if there exists $j=1,...,N$ 
	such that $\varepsilon_{j,2}+\varepsilon_{j+1,1}=0$ (with convention $\varepsilon_{N+1,1}=0$). 
	Thus we restrict the sum in~\eqref{momentN''}  to the sum 
	over the $\varepsilon_1,...,\varepsilon_N$ such that
	for all $j=1,...,N$, 
	such that $\varepsilon_{j,2}+\varepsilon_{j+1,1}\ge 1$.
We call {\bf admissible} any such sequence $\boldsymbol\varepsilon:=(\varepsilon_1,...,\varepsilon_N)$.
Let $\boldsymbol\varepsilon:=(\varepsilon_1,...,\varepsilon_N)$ be an admissible sequence.
\begin{itemize}
	\item We observe that $\#\{j\, :\, \varepsilon_j=(0,0)\}\le N/2$.
	\item The contribution to~\eqref{momentN''} of an admissible sequence $\boldsymbol\varepsilon=(\varepsilon_1,...,\varepsilon_N)$ is
	\[
	D^{(n,N)}_{\varepsilon_1,...,\varepsilon_N}=\mathcal O\left(\sum_{a_1,...,a_N\in\mathbb Z^d}\prod_{j=1}^N\left(|\beta_{a_j}|\, \sum_{k_j=0}^n\Vert \widetilde Q_{k_j,a_{j-1},a_j}^{(\varepsilon_j)}\Vert_{\mathcal B}\right)\right)\, .
	\]
	\item
We observe that there exists $u_0>0$ such that
\begin{equation}\label{comparan}
\sum_{k=0}^n\mathfrak a_k^{-d}=\mathfrak A_n\, ,
\quad
\sum_{k=0}^n
\mathfrak a_k^{-d-2\eta}
=\mathcal O(1)\, ,\quad 
\sum_{k=0}^n
\mathfrak a_k^{-d-\eta}
=\mathcal O\left(\mathfrak A_n^{\frac {1-u_0}2}\right)\, .
\end{equation}
Indeed $d+2\eta=\alpha+\varepsilon>\alpha$. For the last estimate, we use the fact that $(\mathfrak a_k)_k$ is $\frac 1\alpha$-regularly varying, and
infer that
$(\sum_{k=0}^{n}\mathfrak a_k^{-d-\eta})^2$ is either bounded or $2-\frac{2d+2\eta}\alpha$-regularly varying whereas $(\mathfrak A_n)_n$
is $(1-\frac d\alpha)$-regularly varying and diverges to  infinity (and $2-\frac{2d+2\eta}\alpha=1-\frac{d+\varepsilon}\alpha<1-\frac d\alpha$).
\item If,  
for all $j=1,...,N$, $\varepsilon_{j,2}+\varepsilon_{j+1,1}=1$,
then, it follows from Hypothesis~\ref{HHH1} that 
\[
D_{\varepsilon_1,...,\varepsilon_N}^{(n,N)}=
\mathcal O\left(d_{\varepsilon_1,...,\varepsilon_N}^{(n,N)}\right)\, ,\]
with
\begin{equation}\label{defd}
d_{\varepsilon_1,...,\varepsilon_N}^{(n,N)}:=\sum_{a_1,...,a_N\in\mathbb Z^d}\prod_{j=1}^N\left(|\beta_{a_j}|\, |a_j|^{\eta(\varepsilon_{j,2}+\varepsilon_{j+1,1})}\, \sum_{k_j=0}^n\mathfrak a_{k_j}^{-d-\eta\varepsilon_{j,1}-\eta\varepsilon_{j,2}}\right)\, ,
\end{equation}
and so, using~\eqref{comparan}, that
\[
D^{(n,N)}_{\varepsilon_1,...,\varepsilon_N}
=\mathcal O\left(d_{\varepsilon_1,...,\varepsilon_N}^{(n,N)}\right)=\mathcal O\left(\mathfrak A_n^{N_0+
N_1\frac {1-u_0}2
}\right)\, ,
\]
where $N_k:=\#\{j:\varepsilon_{j,1}+\varepsilon_{j,2}=k\}$,
since $\sum_{a\in\mathbb Z^d}(1+|a|)^\eta|\beta_a|<\infty$.
Observe that $N_0+N_1+N_2=N$ and $N=\sum_{k=1}^2\sum_{j=1}^N\varepsilon_{j,k}=N_1+2N_2$ and so $N_2=N_0$
and $N_0=\frac{N-N_1}2$. 
Therefore, in this case, 
\[
D^{(n,N)}_{\varepsilon_1,...,\varepsilon_N}=\mathcal O\left(d_{\varepsilon_1,...,\varepsilon_N}^{(n,N)}\right)=\mathcal O\left(\mathfrak A_n^{\frac{N-u_0N_1}2}\right)=o\left(\mathfrak A_n^{ N/2}\right)
\]
unless if $N_1=0$, i.e. unless if $N$ is even and if $\varepsilon_1,...,\varepsilon_N$ is the alternate sequence $(0,0),(1,1),...,(0,0),(1,1)$.
\item Assume now that there exists some $j_0\in\{1,...,N\}$ such that  $\varepsilon_{j_0,2}+\varepsilon_{j_0+1,1}=2$. 
Recall that it follows from Hypothesis~\ref{HHH1} that, for all $\eta'_{j,1},\eta'_{j,2}\in\{0,\eta\}$,
\begin{equation}
\left\Vert \widetilde Q^{(\varepsilon_j)}_{k,a_j,a_{j-1}}\right\Vert_{\mathcal B}=\mathcal O\left(|a_{j-1}/\mathfrak a_k|^{\eta'_{j,1}\varepsilon_{j,1}}|a_j/\mathfrak a_k|^{\eta'_{j,2}\varepsilon_{j,2}}\mathfrak a_k^{-d
}\right)\, .
\end{equation}
Indeed this follows from~\eqref{normQka},~\eqref{Q'majo1} and~\eqref{Q'majo} combined with the two following facts
\[
\forall\eta\in[0;1],\quad\left\Vert Q''_{k,a,b}\right\Vert_{\mathcal B}=\left\Vert Q'_{k,-a,b-a}-Q'_{k,0,b}\right\Vert_{\mathcal B}=\mathcal O(|b|^\eta)\mathfrak a_k^{-d-\eta}\]
and
\[
\forall\eta\in[0;1],\quad 
\left\Vert Q''_{k,a,b}\right\Vert_{\mathcal B}=\left\Vert Q'_{k,b,b-a}-Q'_{k,0,-a}\right\Vert_{\mathcal B}=\mathcal O(|a|^\eta)\mathfrak a_k^{-d-\eta}\, .\]
	We choose a family $(\eta'_{j,i})_{j=1,...,N;i=1,2}$ of $\{0,1\}$ such that, for all $j=1,...,N$,  $\eta'_{j,2}\varepsilon_{j,2}+\eta'_{j+1,1}\varepsilon_{j+1,1}=\eta=\frac{\alpha+\varepsilon-d}2$ and $\eta'_{1,1}=0$, with convention  $\eta'_{N+1,1}=0$.
	Therefore
	\[
	D_{\varepsilon_1,...,\varepsilon_N}^{(n,N)}=\mathcal O\left(\sum_{a_1,...,a_N\in\mathbb Z^d}\prod_{j=1}^N\left(|\beta_{a_j}|\, |a_j|^{\eta
	}\, \sum_{k_j=0}^n\mathfrak a_{k_j}^{-d-\eta'_{j,1}\varepsilon_{j,1}-\eta'_{j,2}\varepsilon_{j,2}}\right)\right)
	=\mathcal O\left( d_{\varepsilon'_1,...,\varepsilon'_N}^{(n,N)}\right)\, ,
	\]
where we set $\varepsilon'_j:=(\eta'_{j,1}\varepsilon_{j,1},\eta'_{j,2}\varepsilon_{j,2})/\eta$.
We also consider the   sequence $\varepsilon''_1,...,\varepsilon''_N$ obtained from $\boldsymbol\varepsilon'$ by permuting the values
of $\varepsilon'_{j_0,2}$ and $\varepsilon'_{j_0+1,1}$. 
Both sequences $\boldsymbol\varepsilon'$ and $\boldsymbol\varepsilon''$
are admissible and satisfy $\varepsilon'_{j,2}+\varepsilon'_{j+1,1}=\varepsilon''_{j,2}+\varepsilon''_{j+1,1}=1$ for all $j=1,...,N$. Thus it follows from the previous item that
\[
D_{\varepsilon_1,...,\varepsilon_N}^{(n,N)}	=\mathcal O\left( \min\left(d_{\varepsilon'_1,...,\varepsilon'_N}^{(n,N)},d_{\varepsilon''_1,...,\varepsilon''_N}^{(n,N)}\right)\right)=o\left(\mathfrak A_n^{N/2}\right)\, ,
\]
since $\boldsymbol\varepsilon'$ and $\boldsymbol\varepsilon''$ cannot both coincide with the alternate sequence $(0,0),(1,1),...,(0,0),(1,1)$.

\end{itemize}
Estimate~\eqref{dominH} follows from the two last items.

{\bf Acknowledgement~:}
The author conduced this work within the framework of the Henri
Lebesgue Center ANR-11-LABX-0020-01 and is supported by the ANR project GALS (ANR-23-CE40-0001).

\end{appendix}

\end{document}